\documentclass[12pt]{amsart}
\usepackage{amssymb, amstext, amscd, amsmath}
\usepackage{mathtools, xypic, color}
\usepackage{enumitem}
\let\Humlaut=\H
\theoremstyle{plain}
\newtheorem{thm}{Theorem}[section]
\newtheorem{prop}[thm]{Proposition}
\newtheorem{lem}[thm]{Lemma}
\newtheorem{cor}[thm]{Corollary}

\theoremstyle{definition}
\newtheorem{rem}[thm]{Remark}
\newtheorem{defn}[thm]{Definition}
\newtheorem{eg}[thm]{Example}

\mathtoolsset{centercolon}
\newcommand{\bB}{{\mathbb{B}}}
\newcommand{\bC}{{\mathbb{C}}}
\newcommand{\bD}{{\mathbb{D}}}
\newcommand{\bN}{{\mathbb{N}}}
\newcommand{\bR}{{\mathbb{R}}}
\newcommand{\bT}{{\mathbb{T}}}

  \newcommand{\A}{{\mathcal{A}}}

  \newcommand{\F}{{\mathcal{F}}}
  
\renewcommand{\H}{{\mathcal{H}}}
  \newcommand{\I}{{\mathcal{I}}}

  \newcommand{\M}{{\mathcal{M}}}

\renewcommand{\S}{{\mathcal{S}}}

\newcommand{\ep}{\varepsilon}
\renewcommand{\phi}{\varphi}

\newcommand{\Bz}{{\mathbf{z}}}

\newcommand{\bsl}{\setminus}

\newcommand{\dsum}{\displaystyle\sum\limits}
\newcommand{\lip}{\langle}
\newcommand{\rip}{\rangle}
\newcommand{\ip}[1]{\lip #1 \rip}

\newcommand{\ol}{\overline}
\newcommand{\td}{\widetilde}
\newcommand{\wot}{\textsc{wot}}
\newcommand{\AD}{\mathrm{A}(\mathbb{D})}
\newcommand{\Hinf}{H^\infty }
\newcommand{\linf}{\ell^\infty }
\newcommand{\ltwo}{\ell^2}

\newcommand{\AND}{\text{ and }}
\newcommand{\FOR}{\text{ for }}
\newcommand{\FORAL}{\text{ for all }}
\newcommand{\qand}{\quad\text{and}\quad}

\newcommand{\qfor}{\quad\text{for}\quad}
\newcommand{\qforal}{\quad\text{for all}\quad}
\newcommand{\qqand}{\qquad\text{and}\qquad}

\newcommand{\dist}{\operatorname{dist}}
\newcommand{\id}{\operatorname{id}}
\newcommand{\Mult}{\operatorname{Mult}}
\newcommand{\ran}{\operatorname{Ran}}
\newcommand{\re}{\operatorname{Re}}
\newcommand{\spn}{\operatorname{span}}

\begin{document}
\title{Multipliers of embedded discs}

\author[K.R. Davidson]{Kenneth R. Davidson}
\address{Dept.\ of Pure Mathematics, University of Waterloo,
Waterloo, ON, Canada}
\email{krdavids@uwaterloo.ca}
\thanks{The first author is partially supported by an NSERC grant.}

\author[M. Hartz]{Michael Hartz}
\email{mphartz@uwaterloo.ca}
\thanks{The second author is partially supported by an Ontario Trillium Scholarship.}

\author[O.M. Shalit]{Orr Moshe Shalit}
\address{Dept.\ of Mathematics, Ben-Gurion University of the Negev,
Beersheva, Israel}
\email{oshalit@math.bgu.ac.il}
\thanks{The third author is partially supported by ISF Grant no. 474/12, by 
EU FP7/2007-2013 Grant no. 321749, and by GIF Grant no. 2297-2282.6/20.1}
\thanks{The authors thank the Faculty of Natural Science's Distinguished Scientist Visitors Program as well as the Center for Advanced Mathematical Studies at Ben-Gurion University of the Negev, for supporting the first author's visit to Ben-Gurion University during April 2013.}

\begin{abstract}
We consider a number of examples of multiplier algebras on Hilbert spaces
associated to discs embedded into a complex ball in order to examine the isomorphism
problem for multiplier algebras on complete Nevan\-linna-Pick reproducing kernel 
Hilbert spaces.
In particular, we exhibit uncountably many discs in the ball of $\ltwo$ which are 
multiplier biholomorphic but have non-isomorphic multiplier algebras.
We also show that there are \textit{closed} discs in the ball of $\ltwo$ which are varieties,
and examine their multiplier algebras.
In finite balls, we provide a counterpoint to a result of Alpay, Putinar and Vinnikov by 
providing a proper rational biholomorphism of the disc onto a variety $V$ in $\bB_2$ 
such that the multiplier algebra is not all of $\Hinf(V)$.
We also show that the transversality property, which is one of their hypotheses, is
a consequence of the smoothness that they require.
\end{abstract}

\maketitle

\section{Introduction} \label{S:intro}

We are concerned with the multiplier algebras of certain reproducing kernel Hilbert spaces
with the complete Nevan\-linna-Pick property. Using the universal property of the 
Drury-Arveson space, we can identify a variety $V$ in a complex ball $\bB_d$
(where $1 \le d \le \infty$) so that the Hilbert space is a space of analytic functions on $V$.
In this paper, $V$ will usually be homeomorphic to the unit disc.

For each $1 \le d \le \infty$, Drury-Arveson space or symmetric Fock space, $\H_d$, is a space of
analytic functions on $\bB_d$ (where $d=\infty$ corresponds to the unit ball of $\ltwo$) with
the reproducing kernel
\[ k(x,y) = \frac1{1-\ip{x,y}} \qfor x,y \in \bB_d.\]
This is spanned by the vectors $k_y(x) = k(x,y)$ for $y \in \bB_d$.
The multiplier algebra $\M_d$ consists of all (bounded analytic) functions on $\bB_d$ which multiply
$\H_d$ into itself.
It forms a maximal abelian \wot-closed algebra of operators.
Associated to each variety $V$ in $\bB_d$, there is the Hilbert space 
\[ \H_V = \overline{\spn}\{ k_y : y \in V \} . \]
This is considered as a Hilbert space of functions on $V$
with multiplier algebra $\M_V$ (also functions on $V$, not the whole ball). 

It follows that $\M_V$ is completely isometrically isomorphic to
$\M_d/J_V$, where $J_V$ is the ideal of multipliers vanishing on $V$  from \cite{DavPittsPick}(see also \cite{DRS2}).
In particular this means that every multiplier on $V$ extends to a multiplier on the whole ball.
This quotient naturally lives on the zero set of $J_V$.
For this reason, in \cite{DRS2}, we define a variety to be the intersection of zero sets
of multipliers (or of functions in the Hilbert space---see \cite{AM}):
\[
 V = V(\{h_i : i \in \I \}) = \bigcap_{i\in\I} Z(h_i) \cap \bB_d ,
\]
where $Z(h) = h^{-1}(0)$.
Thus there is a technical issue of what a variety should be, as this is not a local property
as in the classical definition of a variety. 

In \cite{DRS2}, we consider the problem of when $\M_V$ and $\M_W$ are isomorphic.
This is completely resolved for (completely) isometric isomorphism.
Here we will be concerned with the question of topological isomorphism.
(Note that since these algebras are semisimple, all algebraic isomorphisms 
are automatically norm continuous.)

The typical way to approach such isomorphism problems regarding algebras of functions
is via the space of characters (non-zero multiplicative linear functionals). 
An isomorphism $\phi:\M_V \to \M_W$ induces a weak-$*$ homeomorphism $\phi^*$
from the character space $M(\M_W)$ onto $M(\M_V)$. 
Moreover there is a natural map $\pi$ of the character space into $\ol{\bB_d}$ obtained
by evaluation on the row contraction $Z := \big[ Z_1 \ \dots \ Z_d \big]$ 
(where $Z_i$ are the multipliers by the coordinate functions $z_i$). Every
point $v \in V$ gives rise to an evaluation functional $\rho_v$ given by $\rho_v(f) = f(v)$. 
The subtlety of our problem stems from the fact that these point evaluation characters 
are typically only a small part of the character space $M(\M_V)$. 

It is shown in \cite[Proposition~3.2]{DRS2} that the points of $\pi(M(\M_V))$ in the open ball are precisely 
the variety $V$, the map is injective on $\pi^{-1}(V) = \{\rho_v : v \in V\}$, 
and these points correspond to the weak-$*$ continuous characters. 
Unfortunately the proof relies on \cite[Theorem 3.2]{DavPitts2}, which states that the characters 
of $\mathfrak{L}_d$ (or $\M_d$) lying in $\pi^{-1}(\bB_d)$ are precisely point evaluations. 
That theorem is not true for $d = \infty$, as we will show. Indeed for every $\lambda\in\bB_\infty$,
the fiber over every point is very large.
Sometimes we are able to work around this, but often we require stronger hypotheses to deal with this issue.
(A new version of \cite{DRS2} corrects this error.)

The main result in \cite{DRS2} about isomorphism is

\begin{thm}[Davidson-Ramsey-Shalit] \label{T:algiso}
Let $V$ and $W$ be varieties in $\bB_d$, with $d<\infty$,
which are the union of finitely many irreducible varieties and a discrete variety. 
Let $\phi$ be a unital algebra isomorphism of $\M_V$ onto $\M_W$. 
Then there exist holomorphic maps $F$ and $G$ from $\bB_d$ into $\bC^d$ 
with coefficients in $\M_d$ such that 
\begin{enumerate}[label=\normalfont{(\arabic*)}]
 \item $F|_{W} = \phi^*|_W \qand  G|_V = (\phi^{-1})^*|_V $
 \item $G \circ F |_W = \id_W \qand F \circ G|_V = \id_V$
 \item $\phi(f) = f \circ F \qfor f \in \M_V$, and
 \item $\phi^{-1}(g) = g \circ G \qfor g \in \M_W$. 
\end{enumerate}
\end{thm}

In particular, when the multiplier algebras are isomorphic, the two
varieties are biholomorphic. 
Since the additional feature is that the component functions of $F$ and $G$ are multipliers,
we will call this a \textit{multiplier biholomorphism}.
Also the isomorphism is always the composition operator induced by this biholomorphism.
The finiteness conditions were needed to establish that $\phi^*$ maps $W$ into $V$, 
instead of possibly sending some component into the corona.

In the case of homogeneous varieties (zero sets of a family of homogeneous polynomials), 
everything works out in the best possible way.
The results of \cite{DRS1,Hartz} combine to show that the multipliers of two
homogeneous varieties are isomorphic if and only if the varieties are biholomorphic.
Moreover the two algebras are similar, and there is a linear map that implements
a (possibly different) biholomorphism between the homogeneous varieties $W$ and $V$.

However, in the non-homogeneous case, a number of examples in \cite{DRS2} showed that
a complete converse to the theorem above is not possible.
One serious issue is that multiplier biholomorphism is not evidently an equivalence relation.
This is because the extension of the maps to the whole ball cannot be composed because
the range is not contained in the ball.
In fact, it is not an equivalence relation at least when the varieties have infinitely many 
components (see Remark~\ref{R:not equivalence}).

There were two types of examples, and we now consider both to have a certain pathology.
The first example concerned Blaschke sequences in the unit disc \cite[Examples 6.2, 8.2]{DRS2}.
The multiplier algebra is isomorphic to $\linf$ if and only if the sequence is an interpolating sequence.
But there are non-interpolating sequences which are biholomorphic to interpolating sequences
in the strong sense that there are $\Hinf$ functions (and even $\AD$ functions) implementing
the bijection. We consider these examples to be somewhat pathological because the variety
has infinitely many components.
See Proposition~\ref{P:Blaschke} for further discussion.

The second class of examples were discs in $\bB_\infty$ \cite[Examples 6.11, 6.12, 6.13]{DRS2}.
The pathological nature has to do with the fact that these are varieties in an infinite dimensional ball.
We shall examine these examples in more detail here. 
In Section \ref{S:Binfty}, we  give precise conditions for when the 
multiplier algebras of two embedded discs in $\bB_\infty$ of a special type are isomorphic.
In particular, we explain when an algebra of this kind is isomorphic to $H^\infty$. 
Our methods allow us to show that there are uncountably many discs which are multiplier biholomorphic
such that their multiplier algebras are not isomorphic.
This will include a family of kernels on the unit disc that lie between Hardy space and Dirichlet space.

Moreover when this family is continued beyond Dirichlet space (section~\ref{S:compact}), we find varieties in
$\bB_\infty$ which are homeomorphic to \textit{closed discs} in the \textit{interior} of the ball.
Again there are uncountably many non-isomorphic multiplier algebras on closed discs which
are all multiplier biholomorphic varieties.
In section~\ref{S:interpolating}, we use interpolating sequences to show that the multiplier 
algebras on many of these compact discs cannot be isomorphic to a multiplier algebra on 
a variety whose closure meets the boundary.
This pathological behaviour seems to depend on the fact that the varieties live in
the infinite dimensional ball.

We shall also be concerned with proper embeddings of discs into finite dimensional balls $\bB_d$.
Here the prototype result is due to Alpay, Putinar and Vinnikov \cite{APV}:

\begin{thm}[Alpay-Putinar-Vinnikov] \label{T:APV}
Suppose that $f$ is an injective holomorphic function of $\bD$ onto 
$V \subset \bB_d$ such that
\begin{enumerate}[label=\normalfont{(\arabic*)}]
\item $f$ extends to an injective $C^2$ function on $\ol{\bD}$,
\item $f'(z) \ne 0$ on $\ol{\bD}$,
\item $\|f(z)\| = 1$ if and only if $|z|=1$,
\item $\ip{f(z),f'(z)} \ne 0$ when $|z|=1$.
\end{enumerate}
Then $\M_V$ is isomorphic to $\Hinf$.
\end{thm}

Note that (3) ensures that the map is proper, (1) implies in particular that
the image of $\ol{\bD}$ is homeomorphic to a closed disc,
while (2) ensures that the inverse map is also holomorphic---so that
this map is a biholomorphism of the two varieties.
Condition (4) is a transversality condition.
We remark that \cite{APV} only asks that $h$ be $C^1$, but in \cite[2.3.6]{ARS}
where this result is generalized to finitely connected planar domains,
they point out that $C^2$ is needed to make the proof work.
This result is further extended in \cite{KMS} to finite Riemann surfaces.

The property that the map is a biholomorphism is clearly necessary.
The kind of difficulty one encounters otherwise is illustrated by the
map $f(z) = \frac1{\sqrt2}(z^2,z^3)$. This map is a proper, bijective, and rational 
map onto a variety $V$ called the Neil parabola.
It extends to be $C^\infty$ on $\ol{\bD}$ and is transversal at the boundary.
However $f'(0) = (0,0)$, so there is a singularity that prevents the inverse
map from being analytic. The multiplier algebra $\M_V = \Hinf(V)$ is naturally identified with
the proper subalgebra of $\Hinf$ consisting of functions $h$ such that $h'(0)=0$.

We will show in Section \ref{S:transversality} that the transversality condition (4) is a consequence of being $C^1$.
A continuous example where transversality fails is presented in Section \ref{S:tangential}.

We will also show (in Section \ref{S:crossing}) that for a minor weakening of the hypotheses of Theorem~\ref{T:APV},
the conclusion is no longer valid.
More precisely, we exhibit a proper rational map $f$ of the disc into $\bB_2$
which satisfies all of the hypotheses except for the fact that the $C^\infty$
extension to $\ol{\bD}$ is not injective, as the boundary crosses itself once,
where the multiplier algebra is not $\Hinf$. 
It is worth noting that this example does not serve as a counterexample to the 
converse of Theorem~\ref{T:algiso}, since the holomorphic inverse 
$f^{-1} : f(\bD) \rightarrow \bD$ is a bounded  analytic function, but not a multiplier. 

In the above mentioned example, $f$ fails to induce an isomorphism between 
$H^\infty$ and $\M_{f(\bD)}$ because $f$ is not injective on $\partial \bD$. 
This begs the question whether the failure can be detected intrinsically in $\bD$. 
In Section \ref{S:pseudo} we show that if a biholomorphism $F: W \rightarrow V$ 
induces an isomorphism $\phi : \M_V \rightarrow \M_W$, then $F$ must be a bi-Lipschitz 
mapping with respect to the pseudohyperbolic distance. 
Re-examination of the example of the preceding paragraph shows that indeed 
$f$ fails to be bi-Lipschitz, hence cannot induce an isomorphism.
Then Theorem \ref{T:pseudo_gen} yields another proof that $\M_V$ is not isomorphic to $H^\infty$. 
However Example~\ref{E:Blaschke} shows that for Blaschke sequences, being bi-Lipschitz 
does not imply isomorphism.

It could possibly be true that a converse to Theorem~\ref{T:algiso} could hold
if the variety has only finitely many irreducible components.
The finiteness of $d$ and the finiteness of the number of components eliminates
all of the counterexamples that we know about.

\section{Multipliers on discs and automorphism invariance}  \label{S:auto_inv}

Let $f:\bD \to V = f(\bD) \subset \bB^d$ be a proper holomorphic map. 
In the case of $d<\infty$, it is well known that if $f$ is injective and $f'(z) \ne 0$ for all $z \in \bD$ then the complex structure
on $V$ as a subset of $\bC^d$ coincides with the complex structure
induced from the homeomorphism with $\bD$. 
We require the analogous result for the case $d=\infty$.

A function $f: \Omega_1 \rightarrow \Omega_2$ between two open balls of two Hilbert spaces is said to be holomorphic if it is Fr\'{e}chet differentiable at every point. Equivalently, $f$ is holomorphic if around every point in $\Omega_1$ there is some neighborhood in which $f$ is represented by a convergent (vector valued) power series.

Suppose that $V,W \subset \ltwo$.
A function $h:V\to \ltwo$ will be called holomorphic if for every $v\in V$, there is a ball
$b_r(v)$ in $\ltwo$ and a holomorphic function $g$ on $b_r(v)$ such that $g|_{V\cap b_r(v)} = h|_{V\cap b_r(v)}$. 
A bijective map $f$ between $V$ and $W$ will be called a biholomorphism provided that 
both $f$ and $f^{-1}$ are holomorphic. 

The following definition is not standard so it is singled out. 

\begin{defn}
We say that a map $f$ from the unit disc into the open unit ball of a Hilbert space is proper if $\lim_{|z|\rightarrow 1}\|f(z)\| = 1$. 
\end{defn}

When the target space is finite dimensional this definition agrees (in this setting) with the standard definition of ``proper map", which is that $f$ is proper if the preimage of every compact set is compact. We require this definition for dealing with maps into infinite dimensional balls. 

The following result is well known when the range is contained in $\bC^d$ for $d<\infty$.
It may well also be known for $d = \infty$, but we have not found this result anywhere.

\begin{prop} \label{P:holomorphic}
Let $f:\bD \to V = f(\bD) \subset \bB_\infty$ be a proper injective holomorphic function such that $f'(z) \ne 0$ for $z \in \bD$. Then $f^{-1}$ is holomorphic.
More generally, a function $h:V\to \bC$ is holomorphic if and only if $h\circ f$ is holomorphic.
\end{prop}

\begin{proof}
Fix $v_0=f(z_0) \in V$. 
As $f'(z_0) \ne 0$,  we can define the rank one projection $P$ onto $\spn\{f'(z_0)\}$. The composed function $P \circ f$ is an analytic function on the disc with nonzero derivative at $0$, hence injective in a neighborhood of $0$. 

We claim that  there is an $r>0$ so that $P$ is injective
on $b_r(v_0) \cap V$. Assume toward a contradiction that $P$ is not injective in any neighbourhood of $v_0$ in $V$. Then there are sequences $w_n$ and $\tilde w_n$ in $V$ which converge to $v_0$ with $w_n \neq \tilde w_n$ and $P w_n = P \tilde w_n$. Write $w_n = f (z_n)$ and $\tilde w_n = f(\tilde z_n)$, and note that $z_n \neq \tilde z_n$. Properness of $f$ implies that  $z_n$ and $\tilde z_n$ are contained in a disc of radius $r < 1$, so by passing to a subsequence, we may assume that $z_n \to z$ and $\tilde z_n \to \tilde z$ for points $z,\tilde z$ in the disc. Thus $f(z) = f(\tilde z) = v_0$. Since $f$ is injective, it follows that $z = \tilde z = z_0$. But $(P f)(z_n) = (P f)(\tilde z_n)$, which contradicts the fact that $P f$ is injective in a neighbourhood of $z_0$. 

Therefore, there is a neighbourhood $b_\ep(z_0) \subset f^{-1}(b_r(v_0))$
so that $Pf|_{b_\ep(z_0)}$ is an injective holomorphic function such that $(Pf)'(z) \ne 0$
for all $z \in b_\ep(z_0)$.
There is an $r_0$ with $0 < r_0 \le r$ so that 
\[ P b_{r_0}(v_0) \subset Pf(b_\ep(z_0)) .\]
Therefore $(Pf|_{b_{\ep}(z_0)})^{-1}$ is a holomorphic function.
If follows that $g = (Pf)^{-1} P$ is a holomorphic function on $b_{r_0}(v_0)$.
It is evident that
\[ g|_{V\cap b_{r_0}(v)} = f^{-1}|_{V\cap b_{r_0}(v)} .\]

Now suppose that $h:V\to\bC$. 
Since the composition of holomorphic functions is holomorphic, 
it follows that if $h$ is holomorphic, then so is $h \circ f$.
Conversely, suppose that $h \circ f$ is holomorphic on $\bD$.
Let $v_0 = f(z_0) \in V$, and let $r_0$ and $\ep$ be as in the previous paragraph.
Since $P|_{b_{r_0}(v_0) \cap V}$ is a homeomorphism onto $b_{r_0}(Pv_0) \cap \ran P$, there is a function 
\[ k: (b_{r_0}(Pv_0) \cap \ran P) \to \bC \]
such that $h = (k \circ P)|_{V\cap b_{r_0}(v)}$.
Now $k = (h \circ f) \circ (Pf)^{-1}$ is holomorphic.
Therefore $\tilde h = k \circ P$ is a holomorphic function on $b_{r_0}(v_0)$
such that $h = \tilde h|_{V\cap b_{r_0}(v)}$.
Thus $h$ is holomorphic.
\end{proof}

We let $\bB_\infty$ denote the open unit ball of $\ltwo$. 
The following result is well known if the range is contained in $\bB_d$ for $d < \infty$.

\begin{cor} \label{C:Hinf(V)}
If $f:\bD \to V = f(\bD) \subset \bB_\infty$ is a biholomorphism, then the space $\Hinf(V)$
of bounded analytic functions on $V$ coincides with $\{ h \circ f^{-1} : h \in \Hinf \}$.
\end{cor}

If $V$ is a variety in $\bB_d$, then the multiplier algebra $\M_V$ of $\H_V$ is a complete quotient of
$\M_d$ by the ideal of multipliers vanishing on $V$. The quotient map is just the restriction map.
Thus every multiplier on $V$ extends to a multiplier on $\bB_d$. 
In particular, they extend to bounded holomorphic functions on the whole ball.
As noted in the introduction, the point evaluation $\rho_v(f) = f(v)$ is always a character on $\M_V$ for each $v\in V$.
Moreover, these are the only point evaluations on $\M_V$, and they coincide with the weak-$*$
continuous characters. 
We will identify the set $V$ with 
\[ \{\rho_v : v \in V\} \subset M(\M_V) . \]
Recall that there is a natural map $\pi$ from the character space of $\M_V$ into $\overline{\bB_d}$ given by $\pi(\rho) = (\rho(Z_i))_{i=1}^\infty$.

An earlier version of this paper relied on \cite[Theorem 3.2]{DavPitts2}, which states that the characters 
of $\mathfrak{L}_d$ (or $\M_d$) lying in $\pi^{-1}(\bB_d)$ are precisely point evaluations. 
That theorem is not true for $d = \infty$, as the following example shows. 
So some additional care is needed.
In the example, we work with the algebra $\M_\infty$ of multipliers on Drury-Arveson space. 

\begin{eg}
Let $(v_n)$ be a sequence in $\bB_\infty$ with the property that $||v_n|| \to 1$, but $(v_n)$ converges weakly to zero. 
By passing to a subsequence, we may assume that $(v_n)$ is interpolating for $\M_\infty$  
(see Proposition~\ref{P:interpolating_sequences}). 
Thus, the unital homomorphism $\Phi: \M_\infty \to \ell^\infty$ defined by $\Phi(f)(n) = f(v_n)$ is surjective, so its adjoint 
$\Phi^*$ is an embedding of the Stone-\v Cech compactification $\beta \bN$ into the character space of $\M_\infty$. 
We claim that every point in $\beta \bN \setminus \bN$ lies in the fiber over the origin, 
i.e., $\pi( \Phi^*(\beta N \setminus N)) = \{0\}$. 
Indeed, let $\phi \in \beta \bN \setminus \bN$. Then for every $k\ge1$, we have 
\[ (\Phi^*(\phi))(M_{z_k}) = \phi(( Z_k(v_n) )) = \lim_{n\to\infty} Z_k(v_n)  = 0 .\]
This shows that there are points in $\pi^{-1}(\bB_d)$ which are not point evaluations.

We can use this construction to show that there are also algebras $\M_V$ with characters that are fibered 
over points in $\bB_d \setminus V$. Let $(v_n)$ be as above, and assume that $v_0 = 0$. 
Let $f \in \M_\infty$ satisfy $f(0)=1$ and $f(v_n)=0$ for $n \ge 1$.
Then $V = f^{-1}(0)$ is a variety such that $0 \notin V$, but the fiber $\pi^{-1}(0)$ is large. 
\end{eg}

Now consider two discs in $\bB_d$, including the case $d=\infty$.
We need a few variants of results in \cite{DRS2}.
Consider two biholomorphisms of discs 
\[ f_i : \bD \to V_i = f_i(\bD) \subset \bB_d \qfor i=1,2 \]
such that $V_i$ are varieties in the sense of \cite{DRS2}; i.e.\ they are the intersection of zero sets of multipliers.
Suppose that $\phi:\M_{V_1} \to \M_{V_2}$ is a continuous algebra homomorphism,
and let $\phi^*$ be the induced map from $M(\M_{V_2})$ to $M(\M_{V_1})$.
Composing this with the evaluation map $\pi$ at the row contraction $Z = \big[ Z_1\ \dots \ Z_d\big]$
yields a map $F_\phi = \pi \circ \phi^*: M(\M_{V_2}) \to \ol{\bB_d}$ given by
\[ F_\phi(\rho) = \phi(Z)(\rho) = \big[ \rho(\phi(Z_i)) \big]_{i=1}^d \qfor \rho \in M(\M_{V_2}) .\]
In particular, $F_\phi|_{V_2}$ maps the variety $V_2$ into $\ol{\bB_d}$.

\begin{thm} \label{T:iso_discs}
Let $V_i$ be discs in $\bB_{d_i}$ as described above.
Furthermore,  assume that
\begin{enumerate}[label=\normalfont{(\arabic*)}]
\item for every $\lambda \in V_1$, the fiber $\pi^{-1}(\lambda) = \{\rho_\lambda\}$, and
\item $\pi (M(\M_{V_1})) \cap \bB_{d_1} = V_1$. 
\end{enumerate}
Let $\phi: \M_{V_1} \to \M_{V_2}$ be a continuous algebra homomorphism.
Then $F = F_\phi|_{V_2}$ is a holomorphic map with multiplier coefficients.
If $F$ is not constant, then $F$ maps $V_2$ into $V_1$.
In this case, $\phi^*|_{V_2} = F$ and $\phi$ is given by composition with $F$, that is,
\[ \phi(h) = h \circ F \qforal h \in \M_{V_1} .\]
In particular, if $\phi$ is injective, then $F$ is not constant.
And if $\phi$ is an isomorphism, $F$ is a biholomorphism of $V_2$ onto $V_1$.
\end{thm}

\begin{rem}
The special hypotheses on the variety $V_1$ always hold when $d_1 < \infty$ by \cite[Proposition 3.2]{DRS2}.
Proposition~\ref{P:condition(2)} below shows that even when $d_1=\infty$, it holds in many cases of interest.
\end{rem}

\begin{proof}
Let $F_i = \phi(Z_i)$ for $1 \le i \le d_1$.  
For $v\in V_2$, let $\rho_v$ denote the character of evaluation at $v$. Then
\[ F(v) = \pi (\phi^*(\rho_v)) = \rho_v(\phi(Z)) = \big[ F_i(v) \big]_{i=1}^{d_1} .\]
Observe that the coefficients $F_i$ are all multipliers.
Since characters are completely contractive, we have
\[ \|F(v)\|^2 = \sum_i |F_i(v)|^2 \le \|Z\|^2 = 1 \qforal v \in V_2 .\]
We claim that $F \circ f_2$ is holomorphic.
If $d_1<\infty$, this is clear since the functions $h_i = F_i \circ f_2$ are.
If $d_1=\infty$, let $\alpha = (a_i)_{i=1}^\infty \in \ltwo$. Then
\[
 \ip{F \circ f_2(z), \alpha} = \sum_{i=1}^\infty \bar a_i h_i(z) .
\]
This converges uniformly on $V_2$ since by the Cauchy-Schwarz inequality,
\begin{align*}
    \sum_{n= N}^\infty |\bar a_n h_n(z)| 
    &\le \Big( \sum_{n=N}^\infty |a_n|^2 \Big)^{1/2}
     \Big( \sum_{n=N}^\infty |h_n(z)|^2 \Big)^{1/2}  \\&\le
    \Big( \sum_{n=N}^\infty |a_n|^2 \Big)^{1/2} \xrightarrow{N \to \infty} 0.
\end{align*}
Therefore $\ip{F \circ f_2(v), \alpha}$ is holomorphic for all $\alpha$, so $F \circ f_2$  is holomorphic.
By Proposition~\ref{P:holomorphic}, $F$ is holomorphic.

Now we assume that $F$ is not constant, and show that $F$ maps into $\bB_d$. 
If $\mu = F(\lambda)$ lies in the boundary $\partial \bB_d$ for some $\lambda \in V_2$,
then $\ip{F \circ f_2(z),\mu}$ is a holomorphic function into $\ol{\bD}$ which takes the value $1$ at $\lambda$. 
By the maximum modulus principle, this function is constant. 
Since the image of $F$ is contained in the closed unit ball, $F \circ f_2$ itself and thus $F$
must be constant.
This contradicts our assumption. 

Now for $v\in V_2$, $\phi^*(\rho_v)$ is fibered over the point $F(v)$, which lies in $\bB_{d_1}$.
By hypotheses (1) and (2), the characters of $\M_{V_1}$ in $\pi^{-1}(\bB_d)$ are precisely the
point evaluations at points of $V_1$. Hence $F$ maps $V_2$ into $V_1$.
Therefore
\[ \phi(h)(v) = \phi^*(\rho_v)(h) = \rho_{F(v)}(h) = h(F(v))  \]
for all $h \in \M(V_1)$ and $v \in V_2$.

If $\phi$ is injective, it follows as in \cite[Lemma 5.4(2)]{DRS2} that $F$ maps $V_2$ into $V_1$.
The argument there assumed that $\phi$ is an isomorphism, but only injectivity is required. 
To recall, suppose that $F$ maps $V_2$ to a single point  $\lambda \in \ol{\bB_d}$. Then for every $i$, we have
\begin{equation*}
  \varphi(\lambda_i - Z_i) = \lambda_i - F_i = 0,
\end{equation*}
hence $Z_i = \lambda_i \in \M_{V_1}$ by injectivity of $\varphi$. This is clearly impossible as $V_1$ consists of
more than one point.
Therefore $F$ is not constant.

Now assume that $\phi$ is an isomorphism.
By an adaptation of \cite[Section 11.3]{DRS1}, the fact that $\phi$ is implemented by composition 
implies that $\phi$ is weak-$*$ continuous. 
Since the closed unit ball $B_1$ of $\M_{V_1}$ is weak-$*$ compact, and since the weak-$*$ topology on 
$\M_{V_2}$ is Hausdorff, $\varphi \big|_{B_1}: B_1 \to \varphi(B_1)$ is a homeomorphism in the weak-$*$ topologies. 
Every bounded set in $\M_{V_2}$ is contained in $r \varphi(B_1)$ for some $r > 0$, 
hence $\varphi^{-1}$ is weak-$*$ continuous on bounded sets.
It follows from the Krein-Smulian theorem that $\phi^{-1}$ is weak-$*$ continuous. 
In particular, $(\phi^{-1})^*$ takes point evaluations to point evaluations. 

We deduce that $\varphi^*(V_2) = V_1$, hence $F$ maps $V_2$ onto $V_1$. Since $F^{-1} = \pi \circ (\varphi^{-1})^*$,
the map $F^{-1}$ is holomorphic with multiplier coefficients.
\end{proof}

\begin{rem} \label{R:extend}
Besides the special assumptions on $\M_{V_1}$, another issue that makes this a weaker result than Theorem~\ref{T:algiso} is that we do not
know if the map $F$ can be extended to the whole ball $\bB_\infty$ to be bounded, or
better yet a bounded multiplier.  Now $F$ is essentially $\phi(Z)$. So if we knew that
$\phi$ was completely bounded, then $F$ would be a bounded multiplier. 
Since $\M_V$ is a complete quotient of $\M_\infty$, we could lift this to a bounded multiplier map on the whole ball.
As it is, we only know that the coordinates are contractive multipliers---so they each extend to
contractive multipliers on the whole ball. 
When $d<\infty$, this then provides the desired extension.
But when $d=\infty$, extending each individual multiplier coefficient does not generally yield a bounded row multiplier.
\end{rem}

Recall that $\A_d$ is the closure in $\M_d$ of the polynomials.

\begin{prop} \label{P:condition(2)}
Suppose that a variety $V$ in $\bB_\infty$ is the intersection of zero sets of a family $\F \subset \A_d$.
Then $\pi (M(\M_V)) \cap \bB_d = V$. 
\end{prop}

\begin{proof}
Since $\M_V \simeq \M_d/J_V$ where $J_V$ is the ideal of multipliers vanishing on $V$, 
every character $\phi$ of $\M_V$ lifts to a character $\psi$ of $\M_d$ that annihilates $J_V$.
Assume that $\lambda\in\bB_\infty$ and $\phi\in\pi^{-1}(\lambda)$, whence $\psi\in\pi^{-1}(\lambda)$ also.
Then $\psi(f)=f(\lambda)$ for every polynomial $f$, and hence for every $f \in \A_d$. 
In particular, as every $f\in \F$ belongs to $J_V$, we have $0 = \psi(f) = f(\lambda)$.
Therefore $\lambda$ belongs to $V$.
\end{proof}

\begin{rem} \label{R:condition(2)}
When the functions $\F$ defining $V$ belong to $\A_d$, they extend to be continuous on the closed ball.
It follows by the same argument that if $\|\lambda\|=1$ and a character $\phi\in\pi^{-1}(\lambda)$,
then $f(\lambda)=0$ for every $f\in\F$. Hence $\lambda \in \bigcap_{f\in\F} f^{-1}(0)$. Thus in the case
where $\bigcap_{f\in\F} f^{-1}(0) = \ol{V}$, we can conclude that $\pi (M(\M_V)) = \ol{V}$. 
This is of interest even when $d<\infty$ (cf. \cite[Corollary 5.4]{KMS}). 
\end{rem}

It is well known that the conformal automorphisms of the unit disc are the
M\"obius maps $\theta = \lambda \big(\frac{z-a}{1-\bar az}\big)$ for $a \in \bD$ and $|\lambda|=1$.
Moreover, the automorphisms of $\Hinf$ are precisely the maps $C_\theta h = h \circ\theta$.
This familiar result is credited to Kakutani in \cite[p.143]{Hoffman}.

If $f:\bD\to V=f(\bD) \subset \bB_d$ is a biholomorphic map onto a variety $V$,
then we can transfer the M\"obius maps to conformal  automorphisms of $V$
by sending $\theta$ to $f\circ \theta \circ f^{-1}$.
Since this can be reversed, these are precisely the conformal automorphisms of $V$.
We say that $\M_V$ is \textit{automorphism invariant} if composition with all of these
conformal maps yield automorphisms of $\M_V$.
A sufficient criterion for automorphism invariance is given in \cite[Theorem 3.5]{CM}. 
For further discussion of this property, the reader is referred to Section 8 in \cite{CM98}.

\begin{cor} \label{C:Mobius}
Let $V_i$ be discs in $\bB_d$ as described above such that $V_1$ satisfies conditions $(1)$ and $(2)$ of Theorem~$\ref{T:iso_discs}$.
Let $\phi: \M_{V_1} \to \M_{V_2}$ be an algebra isomorphism.
Then there is a M\"obius map $\theta$ of $\bD$ such that the following diagram commutes:
\[
    \xymatrix{ \M_{V_1} \ar[r]^{\phi} \ar[d]_{C_{f_1}} & \M_{V_2} \ar[d]^{C_{f_2}} \\ \Hinf \ar[r]_{C_\theta} & \Hinf }
\]
\end{cor}

\begin{proof}
By Theorem~\ref{T:iso_discs}, $F = \phi^*|_{V_2}$ is a biholomorphism of $V_2$ onto $V_1$, and $\phi$ is implemented by composition with $F$.
We will make use of the fact that $\M_{V_i}$ can  be embedded into $\Hinf$ via
\[
  C_{f_i}h = h \circ f_i \qfor h \in \M_{V_i}.
\]
This map is contractive since the multiplier norm on $\M_{V_i}$ dominates the sup norm.
Observe that $\theta = f_1^{-1} \circ F \circ f_2$ is a biholomorphism of $\bD$ onto itself,
and thus is a M\"obius map. Clearly this makes the diagram commute.
\end{proof}

Suppose that the automorphism $\theta$ can be chosen to be the identity or, equivalently, that
$C_F$, where $F=f_1 \circ f_2^{-1}$, is an isomorphism of $\M_{V_1}$ onto $\M_{V_2}$.
Then we will say that $\M_{V_1}$ and $\M_{V_2}$ are \textit{isomorphic via the natural map}.

\begin{cor}  \label{C:mob_inv_iso}
Let $V_i$ be discs in $\bB_d$ as described above such that $V_1$ satisfies conditions $(1)$ and $(2)$ of Theorem~$\ref{T:iso_discs}$.
If $\M_{V_1}$ or $\M_{V_2}$ is automorphism invariant, then $\M_{V_1}$ and $\M_{V_2}$ 
are isomorphic if and only if they are isomorphic via the natural map $C_F$, where $F=f_1\circ f_2^{-1}$.
In particular, if $\M_{V_1}$ is isomorphic to $\Hinf$, then $C_{f_1}$ is implements the isomorphism. \qed
\end{cor}

\section{Transversality}  \label{S:transversality}

Recall that a map of $\bD$ into a ball $\bB_d$ is proper
if $\lim_{|z|\rightarrow 1}\|f(z)\| = 1$. 
If a proper analytic map $f$ of a surface $\S$ into a ball $\bB_d$ extends
to be $C^1$ on $\ol{\S}$, we shall say that the image meets the 
boundary of $\bB_d$ \textit{transversally} at $f(z)$ for $z \in \partial \S$
provided that
\[  \ip{ f(z), f'(z) } \ne 0 .\]

As noted in the introduction, \cite{APV, ARS, KMS} make transversality
at the boundary a hypothesis needed for their results. 
In this section, we show that a proper analytic $C^1$ embedding automatically
meets the boundary transversally. We first consider maps of the unit disc.
Then we provide an extension to finite Riemann surfaces.

\begin{prop} \label{P:transversal}
  Let $f: \bD \to \bB_d$ be an analytic map which extends to be continuous at $1$ such that $||f(1)||=1$. Then
  \begin{equation*}
    \frac{\re \ip{f(1) - f(z),f(1)}}{1 - |z|} \ge  \frac{ 1- | \ip{f(z),f(1)}|}{1-|z|} \ge \frac{1 - |a|}{1 + |a|} > 0
  \end{equation*}
  for all $z \in \bD$, where $a = \ip{f(0),f(1)}$. We have
  \begin{equation*}
     L = \liminf_{z \to 1, z \in \bD} \frac{1 - | \ip{f(z), f(1)}|}{1 - |z|} < \infty
  \end{equation*}
  if and only if the non-tangential limit of
  \begin{equation*}
    \frac{\ip{f(1) - f(z),f(1)}}{1 - z}
  \end{equation*}
  as $z \to 1$ exists. In this case, this limit equals $L$. In particular, if $f$ extends to be differentiable at $1$, then
  $\ip{f(1),f'(1)} > 0$.
\end{prop}

\begin{proof}
  Consider the holomorphic function
  \begin{equation*}
    g: \bD \to \bD, \quad z \mapsto \ip{f(z),f(1)}.
  \end{equation*}
  An application of the Schwarz-Pick lemma (compare the discussion following Corollary 2.40 in \cite{CM}) shows that
  \begin{equation*}
    \frac{1 - |g(z)|}{1-|z|} \ge \frac{1 - |g(0)|}{1 + |g(0)|} \quad \FORAL z \in \bD,
  \end{equation*}
  from which the first claim readily follows.

  The second claim is a direct consequence of the Julia-Carath\'edory theorem \cite[Theorem 2.44]{CM}.
  It follows from the first part that $L > 0$. In particular, if $f$ extends to be differentiable at $1$, then
  \begin{equation*}
    \ip{f'(1), f(1)} = \lim_{z \to 1} \frac{ \ip{f(1) - f(z),f(1)}}{1 - z} = L > 0,
  \end{equation*}
  so that $f$ meets the boundary transversally at $f(1)$.
\end{proof}

The following consequence is immediate.

\begin{cor} \label{C:transversal}
If $f:\bD\to\bB_d$ is a proper analytic map which extends to be $C^1$ on $\ol{\bD}$,
then $f(\bD)$ meets the boundary transversally. 
Indeed, $\ip{ f(z), f'(z)z } > 0 $ for all $z \in \partial \bD$. 
\end{cor}

Here is a generalization of Corollary~\ref{C:transversal} to finite Riemann surfaces. It will not be used in the sequel, but has consequences in the general theory. It shows that the transversality assumptions in \cite{APV,ARS,KMS} are redundant. 

\begin{prop}
Let $\S$ be a finite Riemann surface, and let $f: \S \rightarrow \bB_d$ be a holomorphic map. 
Fix a point $x_0 \in \partial \S$, and assume that $f$ extends to be $C^1$ on $\S \cup \{x_0\}$ and that $f(x_0) \in \partial \bB_d$. 
Then $f(\S)$ meets $\partial \bB_d$ transversally at $f(x_0)$. 
\end{prop}

\begin{proof}
Let $U$ be a neighbourhood of $1 \in \bC$, and let $g$ be a biholomorphism 
from $U$ onto a neighbourhood $V$ of $x_0$ in the double of $\S$ that 
takes $1$ to $x_0$, $\bD \cap U$ to $\S \cap V$ and $\bT \cap U$ to $\partial \S \cap V$.
That such a local parametrization exists follows from Sections 11.2 and 11.3 in \cite{Cohn}. 
Assume without loss of generality that $f(x_0) = e_1 = (1,0 , \ldots, 0)$. 
Denote $W = \bD \cap U$. 
We now consider the map $h = f \circ g: W \to \bB_d$, and our goal is to prove that 
$\re \langle h'(1), h(1) \rangle = \re \langle h'(1), e_1 \rangle \neq 0$. 

To this end, write the first component of $h = (h_1, \ldots, h_d)$ as $h_1 = u + i v$, 
with $u$ and $v$ real and harmonic. 
Now $u$ is harmonic and strictly less than $1$ in $W$, while $u(1) = 1$. 
By Hopf's lemma (see \cite[Lemma 3.4]{GT} or \cite[Lemma 4.3.7]{Han}), 
the outward pointing directional derivative of $u$ at $1$ is positive: meaning simply that 
$\frac{\partial u}{\partial x}(1) > 0$. Thus 
\[
 \re \langle h'(1), h(1) \rangle = 
 \re (\frac{\partial u}{\partial x}(1) + i \frac{\partial v}{\partial x}(1)) = \frac{\partial u}{\partial x}(1) 
 > 0 ,
\]
as required. 
\end{proof}

Let us examine the geometric meaning of Corollary \ref{C:transversal}. For every $n$ (including $n=1$) the space $\bC^n$ carries the structure of a $2n$-dimensional real Hilbert space with inner product 
\[
\ip{u,v}_\bR = \re \ip{u,v} . 
\]

Let $f$ be as in the corollary, and let us assume for brevity that $f$ extends analytically to a disc $(1+\epsilon)\bD$. The derivative $f'(z)$ is a linear map from the complex tangent space of $\bC$ at $z$ (which can be identified with $\bC$) into the complex tangent space of $f((1+\epsilon)\bD)$ at $f(z)$ (which can be identified with a subspace of $\bC^d$ of complex dimension $1$). Every $z \in \partial \bD$ also serves as the outward pointing normal vector of the real submanifold $\partial \bD$ at the point $z$. The derivative $f'(z)$ maps $z$ to the vector $f'(z)z$. 

Intuitively, a curve $f(\ol{\bD})$ is transversal to $\partial \bB_d$ at $f(z)$ (for $z \in \partial \bD$) if the real valued inner product of the tangent vector to the curve at $f(z)$ with the outward pointing normal vector at $f(z)$ is positive. But since the outward pointing normal of $\partial \bB_d$ at $f(z)$ is (colinear with) $f(z)$, this boils down to the condition $\ip{f(z), f'(z)z}_\bR = \re \ip{f(z), f'(z)z} > 0$. Corollary \ref{C:transversal} gives slightly more information. 

The following proposition and corollary clarify further the geometric meaning of Proposition~\ref{P:transversal} and Corollary~\ref{C:transversal}. 

\begin{prop}\label{P:non-tangential}
Let $\phi$ be a differentiable map from the interval $[0,1]$ into the closed unit ball $B$ of a real Hilbert space  such that $\|\phi(1)\| = 1$ and $\ip{\phi'(1), \phi(1)} > 0$. Then for $x$ near $1$ 
\[
\|\phi(1) - \phi(x)\| \sim 1 - \|\phi(x)\| \sim 1 - x. 
\]
\end{prop}
Here we use the notation $a(x) \sim b(x)$ to mean $\lim_{x \rightarrow 1} \frac{a(x)}{b(x)} = c \in (0,\infty)$. 
\begin{proof}
By differentiability 
\[
\|\phi(1) - \phi(x)\| = \|\phi'(1)(x-1) + o(1-x)\| \sim 1-x, 
\]
since $\phi'(1) \neq 0$. Moreover $1-\|\phi(x)\| \sim 1 - \|\phi(x)\|^2$ and 
\[
1 - \|\phi(x)\|^2 = 1 - \|\phi(1)+\phi'(1)(x-1) + o(x-1)\|^2 = 2\ip{\phi'(1),\phi(1)}(1-x) + o(1-x), 
\]
and the latter is $\sim 1-x$. 
\end{proof}

\begin{cor} \label{C:non-tangential}
Suppose that $f$ is a proper analytic map of $\bD$ into a ball $\bB_d$,
and that $f$ extends to $\bD \cup \{1\}$ and is differentiable at $1$.
Then there exist $c > 0$ such that for all $x \in (0,1)$, 
\[
c \leq \frac{\dist(f(x), \partial \bB_n)}{\| f(1) - f(x) \|} = \frac{1 - \|f(x)\|}{\| f(1) - f(x) \|} \leq 1. 
\]
\end{cor}

\section{Tangential embedding}  \label{S:tangential}

Following the discussion in the previous section we ask: can a proper biholomorphic embedding of the disc into the ball that extends continuously to the boundary meet the sphere tangentially? 
Proposition~\ref{P:transversal} shows that $\ip{f'(1),f(1)}$ is always bounded away from $0$, 
when $f$ extends to be differentiable at $1$. And the Julia-Caratheodory Theorem shows that 
differentiability (at least in the direction of $f(1)$) 
is equivalent to having a bounded differential quotient along some approach to the boundary point.
So a possible reformulation of a tangential condition might be that
\[ \lim_{x \to 1, x \in (0,1)} \frac{\re\ip{f(1) - f(x),f(1)}}{1-x} = +\infty .\] 

A different formulation is used in \cite{ARS}.
They suggest that the tangential condition should be 
\[
 \liminf_{x\to 1,\,x\in(0,1)} \frac{\dist(f(x), \partial \bB_n)}{\| f(1) - f(x) \|} = 
 \liminf_{x\to 1,\,x\in(0,1)} \frac{1 - \|f(x)\|}{\| f(1) - f(x) \|} = 0 .
\]
If this is an actual limit, this intuitively says that as $x$ approaches $1$ along the real axis, 
the curve $f(x)$ approaches the boundary much more quickly than it approaches $f(1)$, and hence
must approach $f(1)$ along a curve tangent to the boundary. 

Corollary~\ref{C:non-tangential} shows that if $f$ is holomorphic and differentiable at $1$,
then the curve $f(x)$ cannot approach $\partial \bB_d$ tangentially in either of these senses. 
We have been unable to determine a relationship between these two tangential conditions.

We now construct an example of a continuous proper embedding of a disc
into $\bB_2$ which meets the boundary tangentially in both of these senses.
Unfortunately we have been unable to determine whether the multiplier algebra is isomorphic to $\Hinf$.

\begin{eg}
The following is a modification of an example shown to us by Josip Globevnik.
There is a proper embedding $F$ of $\bD$ into $\bB_2$ which extends to be continuous
on $\ol{\bD}$ such that
\[ \lim_{x\to 1,\,x\in(0,1)} \frac{1 - \|F(x)\|}{\| F(1) - F(x) \|} = 0 ,\]
and
\[ \lim_{x\to 1,\,x\in(0,1)} \frac{\re \ip{F(1)-F(x), F(1)}} {1-x} = +\infty .\]

Let $A$ be the region in the upper half plane bounded by two semicircles 
in the upper half of the unit disc which are tangent at $1$, 
and have radii $r_1 = \frac12$ and $r_2 =\frac34$ 
together with the line segment $[-\frac12,0]$.
Let $f$ be a conformal map of $\bD$ onto $A$ such that $f(1)=1$. For definiteness, we may assume that 
$f(-i)=0$ and $f(i) = -\frac12$ . 

The map $f$ can be achieved by the following sequence of conformal maps. 
First apply the M\"obius map $w \to \frac{w+i}{iw+1}$ which takes $-i$ to 0, $i$ to $\infty$,
carries $\bD$ onto the upper half plane, takes $1$ to $1$, and is analytic in a neighbourhood of $1$.
Then take the square root map onto the first quadrant, followed by the M\"obius map
$w \to \frac{w-1}{w+1}$ which carries the quadrant onto the upper half disc. 
Call the composition of these maps $g$. 
Then $g$ maps the disc onto the upper half disc, $g$ takes $1$ to $0$, 
and is still analytic in a neighbourhood of $1$; and $g(\pm i) = \pm 1$. 
Now the standard branch of $\log$ (with cut along the negative imaginary axis) 
carries the region onto the half strip bounded by the negative real axis $(-\infty,0]$, 
the line segment $[0,\pi i]$ and half line $(-\infty, \pi i]$ parallel to the real line. 
Then take a final M\"obius map $w \to \frac{w-\pi i}{w+ 2\pi i}$. 
The composition of all these maps is the desired map $f$.

Observe that $f$ extends to a homeomorphism of $\ol{\bD}$ onto $\ol{A}$ and satisfies $f(1)=1$.
The map $g$ from $\bD$ to the half circle is conformal in a neighbourhood of $1$, so 
$g(e^{it}) \approx at$ where $g'(1) = -ia \ne 0$; in fact, $a = \frac14$.
Hence $\log g(e^{it}) \approx \log(at)$ for $t>0$ and $\log g(e^{it}) \approx \log(a|t|) + \pi i$ for $t<0$.
So we obtain that 
\[
 f(e^{it}) \approx \begin{cases}
                               \frac{ \log(at) - \pi i}{\log(at) + 2\pi i} &\FOR t>0\\[1ex]
                               \frac{ \log(a|t|) }{\log(a|t|) + 3\pi i} &\FOR t < 0 
                           \end{cases}
\]
Hence we may compute that
\[
 u(e^{it}) := \frac12 \log \big(1 - |f(e^{it})|^2 \big) \approx  -\log \log |t|^{-1} .
\]
In particular, $u$ is in $L^1(\bT)$. 

Fix $2/3 < r < 1$, and define $\rho(z) = rz + 1-r$. 
This maps $\bD$  onto a disc of radius $r$ tangent to $\bD$ at $1$.
Therefore $f_1(z) = f(\rho(z))$ maps $\bD$ conformally onto a region contained in $A$
which extends to be analytic on a neighbourhood of $\ol{\bD}\bsl\{1\}$.
It is still true that 
\[
 u_1(e^{it}) := \frac12 \log \big(1 - |f_1(e^{it})|^2 \big)
\]
belongs to $L^1$, but now it is $C^\infty$ except at $1$, where it goes to $-\infty$.
Hence $u_1$ extends to a real harmonic function on $\ol{\bD}$
which is smooth except at $1$, where it goes to $-\infty$. 
Let $\tilde u_1$ be its harmonic conjugate.
This is also smooth except at $1$.
Let $f_2(z) = e^{u_1 + i \tilde u_1}$.
Then $f_2$ extends to be continuous on $\ol{\bD}$ with $f_2(1) = 0$,
and $f_2$ is smooth except at $1$.

Now $|f_1(e^{it})|^2 + |f_2(e^{it})|^2 = 1$ on $\bT$.
It follows that $F(z) = (f_1(z), f_2(z))$ is a proper map of $\bD$ into $\bB_2$
that extends to be continuous on $\ol{\bD}$, and smooth except at $1$.
Since $f_1$ is conformal, $F$ is a biholomorphism of $\bD$ onto its image.

It is easy to see that as $z$ approaches $1$, $F(z)$ approaches $(1,0)$ tangentially.
Thus it follows that
\[
 \lim_{x\to 1,\,x\in(0,1)} \frac{1 - \|F(x)\|}{\| F(1) - F(x) \|} = 0 .
\]

A careful look at the estimates above shows that for $x \in (0,1)$,
\[
 f(1-x) \sim \frac{\log(ax) - \frac\pi 2 i}{\log(ax) + \frac{3\pi}2 i} 
 \sim (1 - \frac{c_1}{\log^2 x}) + i \frac{c_2}{\log x} .
\]
Hence
\[ \re \ip{ f(1) - f(x), f(1) } \sim \frac{c_1}{\log^2(1-x)} ,\]
so that
\[
  \lim_{x \to 1,\, x\in(0,1)}  \frac{\re \ip{F(1)-F(x), F(1)}} {1-x}= +\infty .
\]  
\end{eg}
\medskip

\section{Crossing on the boundary}  \label{S:crossing}

In this section, we will provide a method for constructing a smooth proper embedding of a disc into a ball
such that the multiplier algebra is not all of $\Hinf$.
The idea is to have the boundary cross itself.

\begin{thm} \label{T:notHinf}
Suppose that $f:\bD \to \bB_d$ is a proper analytic map which satisfies
\begin{enumerate}[label=\normalfont{(\arabic*)}]
 \item $f|_\bD$ is injective,
 \item $f$ extends to a differentiable map on $\bD \cup \{\pm1\}$, and
 \item $f(1)=f(-1)$.
\end{enumerate}
Suppose that $V = f(\bD)$ is a variety $($in the sense of \cite{DRS2}$)$.
Then $f^{-1} \not \in \M_V$. In particular, the embedding
\begin{equation*}
  \M_V \to H^\infty, \quad h \mapsto h \circ f,
\end{equation*}
is not surjective.
\end{thm}

\begin{proof}
We first make some first order estimates in order to approximate
the kernel functions near $\pm1$.
By Proposition~\ref{P:transversal}, we have $\ip{f'(1),f(1)} > 0$.
Furthermore, differentiability of $f$ at $1$ implies that for small $x > 0$, we have
\[  f(1-x) = f(1) - x f'(1) + o(|x|) .\]
Hence
\begin{align*}
 1-\| f(1-x)\|^2 &= \|f(1)\|^2 - \|f(1-x)\|^2 \\
 &= \ip{f(1), f(1) - f(1-x)} + \ip{x f'(1) + o(|x|), f(1) - x f'(1) + o(|x|) }\\
 &= 2  x\ip{f'(1), f(1)} + o(|x|) .
\end{align*}
Similarly, $\ip{f'(-1),f(-1)} < 0$; and for small $y$ with $y > 0$,
\[
 f(-1+y) = f(-1) + y f'(-1) + o(y) 
\]
and
\[ 
 1-\| f(-1+y)\|^2 = -2 y\ip{f'(-1), f(-1)} + o(y) .
\]

Likewise, for small positive values of $x$ and $y$, we obtain (using $f(1) = f(-1)$)
\begin{align*}		
 1 &- \ip{f(1-x), f(-1+y)} \\
 &= 1 - \ip{f(1) - x f'(1) + o(x), f(-1) + y f'(-1) +o(y)}  \\
 &= 1 - \ip{f(1), f(-1)} - \ip{ f(1), yf'(-1)} + \ip{xf'(1) , f(-1) } + o(x+y) \\
 &= x\ip{f'(1), f(1)} - y\ip{f'(-1),f(-1)} + o(x+y) .
\end{align*}

Choose the positive scalar $s$ so that 
\[ 0 < a := \ip{f'(1),f(1)} = -s \ip{f'(-1),f(-1)} .\]
Then set $y=sx$.
We have that
\begin{align*}\label{Eq1} \tag{\dag}
 \frac{(1-\|f(1-x)\|^2)(1-\|f(-1+sx)\|^2)}{\big| 1- \ip{f(1-x), f(-1+sx)} \big|^2} 
 &= \frac{(2ax + o(x))(2ax + o(x))}{(2ax + o(x))^2} 
 = 1 + o(1) .
\end{align*}

Let $\|h\|_M$ denote the multiplier norm in $\M_V$. 
Assume for a contradiction that $f^{-1}$ is a multiplier. Set $C = ||f^{-1}||_M$ and
$h = f^{-1} / C$, so that
$\|h\|_M = 1$.

We apply the Pick condition to this $h$ at the points $\{f(1-x),f(-1+sx)\}$:
\begin{align*}
 0 &\le 
 \begin{bmatrix}
  \frac{1 - |h(f(1-x))|^2}{1-\| f(1-x)\|^2} & \frac{1 - h(f(1-x)) \ol{h(f(-1+sx))}}{1 - \ip{f(1-x), f(-1+sx)}} \\[2ex]
  \frac{1 - h(f(-1+sx)) \ol{h(f(1-x))}}{1 - \ip{f(-1+sx), f(1-x)}} & \frac{1 - |h(f(-1+sx))|^2}{1-\| f(-1+sx)\|^2}
 \end{bmatrix}
 \\[1ex] &=
 \begin{bmatrix}
  \frac{1 - C^{-2}(1-x)^2}{1-\| f(1-x)\|^2} & \frac{1 + C^{-2} (1-x)(1-sx)}{1 - \ip{f(1-x), f(-1+sx)}} \\[2ex]
  \frac{1 + C^{-2} (1-x)(1-sx)}{1 - \ip{f(-1+sx), f(1-x)}} & \frac{1 - C^{-2} (1-sx)^2}{1-\| f(-1+sx)\|^2}
 \end{bmatrix}.
\end{align*}
Hence the determinant is positive. Clearing the denominators yields
\begin{align*} 
 \big(C^2 +&\, (1 - x)(1 - sx) \big)^2  \big(1 - \| f(1-x)\|^2\big) \big( 1 - \| f(-1+sx)\|^2\big) \\&\le
 \big(C^2 - (1-x)^2)(C^2 - (1-sx)^2\big) \big|1 - \ip{f(-1+sx), f(1-x)} \big|^2 . 
\end{align*}
Using the estimate from (\ref{Eq1})  and letting $x$ decrease to $0$, we obtain
\[
 (C^2 + 1)^2  \le (C^2 - 1)^2   . 
\]
As this is false, we deduce that $f^{-1} \not \in \M_V$.
\end{proof}

Now we show that a map with these properties can be obtained.

\begin{thm} \label{T:embed}
There is a a rational function $f$ with poles off $\ol{\bD}$ and values in $\bC^2$  which
satisfies the conditions of Theorem~$\ref{T:notHinf}$, meets $\partial \bB_2$
transversally, and is one-to-one except for the fact that $f(-1)=f(1)$, and
so that $f$ is a biholomorphism.
Then $V = f(\bD)$ is a variety $($in the sense of \cite{DRS2}$)$ such that 
$\M_V \subsetneqq \Hinf(V)$.
In particular, $f^{-1}$ is not a multiplier.
\end{thm}

\begin{proof}
Fix $0 < r < 1$, and let 
\[ b(z) = \frac{z-r}{1-rz} .\]
Note that $b(\pm 1) = \pm 1$. Define
\[
 f(z) = \frac1{\sqrt2} \big( z^2, b(z)^2 \big) .
\]
Then it is clear that $f$ is a rational function with poles off of $\ol{\bD}$.
Since $z$ and $b(z)$ are automorphisms of the disc, it is easy to see that 
$\|f(z)\|=1$ when $|z|=1$.
So $f(\bD)$ is contained in the ball $\bB_2$.

Since $f$ is analytic on a disc $(1+\ep)\bD$ for some $\ep>0$ and 
\[ V = f((1+\ep)\bD) \cap \bB_2 ,\]
it follows that $V$ is a variety \cite{KMS}.
By Proposition~\ref{P:transversal}, $V$ meets the boundary transversally at every point.

Note that the first coordinate of $f(z)$ is $z^2/\sqrt2$.
Hence if $f(w)=f(z)$, we have $w=\pm z$. 
So equality implies that $b(-z)^2 = b(z)^2$, which
is easily seen to have solutions $z \in\{0,\pm1\}$.
Thus $f(-1)=f(1)$ is the only failure to be one-to-one.
Moreover, 
\[ f'(z) = (2z,2 b(z) b'(z)) \]
is never zero since the first coordinate vanishes only at $z=0$, 
while
\[ 2b(0)b'(0) = -2r(1-r^2) \ne 0 .\]
So this map is a biholomorphism.
It is now clear that the hypotheses of Theorem~\ref{T:notHinf} are satisfied.
Therefore, $\M_V \subsetneqq \Hinf(V)$ and indeed, $f^{-1}$ is not a multiplier.
\end{proof}

\begin{rem}
The fact that $f^{-1}$ is not a multiplier means that this approach will not yield
counterexamples to the converse of Theorem~\ref{T:algiso}. 
\end{rem}

Corollary~\ref{C:mob_inv_iso} and the automorphism invariance of $\Hinf$
yields the following consequence.

\begin{cor}
For $V$ given in Theorem~$\ref{T:embed}$, $\M_V$ is not isomorphic to $\Hinf$.
\end{cor}

\section{Pseudohyperbolic distance}  \label{S:pseudo}

Recall that in the ball $\bB_d$, we can define the pseudohyperbolic distance
\[ d(z,w) = \| \phi_w(z) \| = \| \phi_z(w) \| \]
where $\phi_w$ is the conformal automorphism of $\bB_d$ onto itself interchanging
the points $w$ and $0$ given by
\[ \phi_w(z) = \frac{ w - P_w z - (1-\|w\|^2)^{1/2} P_w^\perp z}{1-\ip{z,w}} \]
where $P_w$ is the orthogonal projection of $\bC^d$ onto $\bC w$.
By \cite[Theorem 2.2.2(iv)]{Rudin}, 
\[
 d(z,w)^2 = 1 - \frac{(1-\|w\|^2)(1-\|z\|^2)}{|1-\ip{w,z}|^2} = 1 - |\ip{\nu_w,\nu_z}|^2
\]
where $\nu_w = (1-\|w\|^2)^{1/2} k_w$ is the normalized reproducing kernel for Drury-Arveson space.

The Schwarz lemma \cite[Theorem 8.1.4]{Rudin} in this context states that if $F$ is a holomorphic map of
$\bB_d$ into $\bB_e$, then
\[ d(F(z), F(w)) \le d(z,w) .\]

\begin{lem}
For every $\lambda, \mu \in V$, 
\[
 d(\lambda,\mu) \leq \|\rho_\lambda - \rho_\mu\| \leq 2 d(\lambda,\mu) .
\]
\end{lem}

\begin{proof}
The inequality $\|\rho_\lambda - \rho_\mu\| \leq 2 d(\lambda,\mu)$ was observed in \cite[Lemma 5.3]{DRS2}. 
For completeness, by the Schwarz lemma 
\[
\left|\frac{f(\lambda) - f(\mu)}{1 - f(\lambda)\ol{f(\mu)}} \right| \leq d(\lambda, \mu) ,
\]
for all $f$ with $\|f\| \leq 1$ and it follows that 
\[
 \|\rho_\lambda - \rho_\mu\| \leq 
 d(\lambda,\mu) \sup_{\|f\| \leq 1} |1 - f(\lambda) \ol{f(\mu)}| 
 \leq 2 d(\lambda, \mu) . 
\]
For the lower bound we may assume (applying an automorphism $\phi$ 
and replacing $V$  with $\phi(V)$, this induces an isometry on the characters 
as well as on the points) that $\mu = 0$. 
But then clearly $f(z) = \langle z, \lambda/|\lambda| \rangle$ is a multiplier of norm $1$ 
such that 
\[ |\rho_\lambda(f) - \rho_0(f)| = |\lambda| = d(\lambda,\mu) . \qedhere \]
\end{proof}

Note that a biholomorphism $F: W \to V$, being a homeomorphism, is automatically proper.

\begin{thm}\label{T:pseudo_gen}
Suppose that $\phi : \M_V \rightarrow \M_W$ is an isomorphism induced by 
a biholomorphism $F : W \rightarrow V$. 
Then there are constants $c,C>0$ such that 
\[
 c\,d(\lambda, \mu) \leq d(F(\lambda), F(\mu)) \leq C\,d(\lambda, \mu)
 \qforal \lambda, \mu \in W. 
\]
\end{thm}

\begin{proof}
Put $t = \|\phi^{-1}\|^{-1}$, and denote by $(\M_V)_1$ and $(\M_W)_1$ 
the unit balls of $\M_V$ and $\M_W$. 
Then $t \cdot (\M_W)_1 \subseteq \phi((\M_V)_1)$, so  
\begin{align*}
\|\rho_{F(\lambda)} - \rho_{F(\mu)}\| &= \sup_{f \in (\M_V)_1} |\phi(f)(\lambda) - \phi(f)(\mu)| \\
& = \sup_{g \in \phi((\M_V)_1)}|g(\lambda) - g(\mu)| \\
& \geq \sup_{g \in (\M_W)_1}|tg(\lambda) - tg(\mu)| \\
& =t \|\rho_\lambda - \rho_\mu\|.
\end{align*}
Applying the lemma gives 
\[
 t\cdot d(\lambda,\mu) \leq t\cdot \|\rho_\lambda - \rho_\mu\| 
 \leq  \|\rho_{F(\lambda)} - \rho_{F(\mu)}\| \leq 2  d(F(\lambda), F(\mu)).
\]
This gives one inequality with $c = t/2$. The other inequality follows by symmetry. 
\end{proof}

\begin{rem}
The proof of Theorem~\ref{T:notHinf} shows that
\[
 d(f(1-x), f(-1+sx))^2 = 1 - \frac{(1-\|f(1-x)\|^2)(1-\|f(-1+sx)\|^2)}{|1-\ip{f(1-x), f(-1+sx)}|} = o(1) .
\]
That is, we have
\[
 \lim_{x\to 0^+} d(f(1-x), f(-1+sx)) = 0.
\]
It follows that $f$ does not induce an isomorphism between $\M_V$ and $\Hinf$. 

Moreover in the example in Theorem~\ref{T:embed}, an easy estimate shows that
$\|f'(z)\| \ge \sqrt2$ on $\partial \bD$. Since $f'$ never vanishes, we have that
$\inf_{z\in\bD} \|f'(z)\| > 0$. Nevertheless, because of the crossing on the boundary,
the previous paragraph shows that the pseudohyperbolic distance is not preserved
up to a constant. Thus this property is not just a local condition.
\end{rem} 

\begin{eg}
Consider the sequences 
\[ v_n = 1-1/n^2 \qand  w_n = 1-e^{-n^2} \qfor n \ge 1 ,\]
and set $V = \{v_n\}_{n=1}^\infty$ and $W = \{w_n\}_{n=1}^\infty$. 
In \cite[Example 6.2]{DRS2} these two varieties were examined, 
and it was shown that there exist $g,h \in \Hinf$ such that 
\[ h \circ g|_V = \id_V \qand g \circ h |_W = \id_W ,\]
while at the same time, since $W$ is interpolating and $V$ is not, $\M_V$ and $\M_W$ are not isomorphic. 
Theorem \ref{T:pseudo_gen} sheds new light on this example. Indeed, we can check that 
\[ d(v_n, v_{n+1}) = \frac{2n + 1}{2n^2 + 2n} \rightarrow 0, \]
while
\[ d(w_n, w_{n+1}) = \frac{1-e^{-2n-1}}{1+e^{-2n-1}-e^{-n^2-2n-1}} \rightarrow 1 .\]
Thus the biholomorphisms $g$ and $h$ are not bi-Lipschitz on the varieties, hence they cannot induce an isomorphism. 
\end{eg}

The following result generalizes this example significantly.

\begin{prop} \label{P:Blaschke}
Let $V = \{v_n\}$ be a Blaschke sequence in $\bD$.
Then there is an interpolating sequence $W = \{w_n\}$ and functions $g,h \in \Hinf$
such that
\[ g(v_n) = w_n  \qand h(w_n) = v_n \qforal n\ge 1 .\]
\end{prop}

\begin{proof}
Let $b_a(z) = \frac{\bar a}{|a|} \frac{a-z}{1-\bar a z}$ for $a \in \bD$.
Define
\[ \delta_n := \prod_{i \ne n} |b_{v_i}(v_n)| .\]
These values are positive because $V$ is a Blaschke sequence.
(Carleson's interpolation theorem shows that $V$ is an interpolating sequence
if and only if it is strongly separated, i.e.\ $\inf_{n\ge1} \delta_n > 0$.) 
A result of Garnett \cite[Theorem 4]{Gar77} (see \cite[ch.VII, Exercise 9]{Garnett})
shows that if 
\[ |a_n| \le \delta_n (1 + \log \delta_n^{-1})^{-2} , \]
there is an $f\in \Hinf$ such that $f(v_n) = a_n$ for $n \ge 1$.
Choose $(a_n)$ with $a_n>0$ satisfying these inequalities, and tending to $0$ sufficiently fast that
$w_n = 1- a_n$ is an interpolating sequence.
Then $g = 1-f$ is the desired map.
Since $W$ is an interpolating sequence, there is is an $h \in \Hinf$ such that $h(w_n) = v_n$ for all $n\ge1$.  
\end{proof}

It is tempting to conjecture that a biholomorphism with multiplier coordinates between two varieties, 
which is also bi-Lipschitz with respect to the pseudohyperbolic distance $d$, induces an isomorphism. 
The following example shows that this fails.  

\begin{eg} \label{E:Blaschke}
A Blaschke sequence $V = \{v_n\}$ is \textit{separated} if 
\[\inf_{m\neq n} d(v_m,v_n) > 0 .\]
Interpolating sequences are separated, and are characterized by being strongly separated.
However there are Blaschke sequences which are separated but not strongly separated,
and thus are not interpolating.
For such a sequence $V$, the maps constructed in Proposition~\ref{P:Blaschke}
will be bi-Lipschitz in the pseudohyperbolic metric but the multiplier algebras
are not isomorphic.

An explicit example of a separated but not interpolating sequence is given in \cite{DSV}.
Here is a related example which has the additional virtue of having $1$ as the only limit point
of the sequence. Let
\[ v_{n,k} = (1-2^{-n}) e^{ik2^{-n}} \qfor n \ge 1 \AND 0 \le k < 2^{n/2} .\]
Then set $V = \{v_{n,k}: n \ge 1,\ 0 \le k < 2^{n/2} \}$.
It is routine to verify that this satisfies the Blaschke condition and is separated.
In order for the sequence to be interpolating, it is necessary that the measure
$\mu = \sum_{n,k} (1-|v_{n,k}|) \delta_{v_{n,k}}$ be a Carleson measure \cite[Theorem VII.1.1]{Garnett}.   
This means that there is a constant $C$ so that $\mu(S(I)) \le C |I|$ for every arc $I \subset \bT$,
where
\[ S(I) = \{ re^{i\theta}: 1-|I| \le r < 1,\ e^{i\theta} \in I \} .\]
But $\mu$ is not a Carleson measure: for $p\ge1$, let 
\[ S_p = S([0,2^{-p})) = \{ re^{i\theta}: 1-2^{-p} \le r < 1,\ 0 \le \theta < 2^{-p} \} .\]
Then
\begin{align*}
 \frac1{2^{-p}} \sum_{v_{n,k}\in S_p} 1-|v_{n,k}| &= 2^p \sum_{n\ge p} 2^{-n} \min\{2^{n-p}, 2^{n/2} \} \\
 &\ge 2^p \sum_{n=p}^{2p} 2^{-n} 2^{n-p} = p+1 .
\end{align*}
This is not bounded.
\end{eg}

\begin{rem} \label{R:not equivalence}
Proposition~\ref{P:Blaschke} raises a fundamental issue in finding a converse to Theorem~\ref{T:algiso}.
The property of having a multiplier biholomorphism between two varieties $V$ and $W$ 
\textit{is not an equivalence relation}!
The proposition shows that every Blaschke is equivalent to some interpolating sequence.
Moreover, examination of the proof shows that if $V = \{ v_n\}$ and $X = \{ x_n\}$ are Blaschke sequences,
there is a common interpolating sequence $W$ that works for both.

However in general, there is no $h\in \Hinf$ such that $h(V) = X$.
To see this, let 
\[ v_n = 1 - n^{-2} \qand  x_n = (-1)^n v_n \qfor n \ge 2 .\]
Suppose that $h$ exists. Let $C=\|h\|_\infty$ and $g = C^{-1}h$.
Then there is an increasing sequence $n_i$ so that
$h(v_{n_i}) > 1/2$ and $h(v_{n_i + 1}) < -1/2$. 
Then $d(v_{n_i} ,v_{n_i + 1}) \approx 1/n_i$ but 
\[ d(g(v_{n_i}), g(v_{n_i + 1}) ) > d(\tfrac1{2C}, \tfrac{-1}{2C} ) =: \delta > 0 .\]
This contradicts the Schwarz inequality.

The problem is that we cannot compose these maps \textit{on the whole ball}.
The extensions of $g$ and $h$ to the ball do not have norm 1, and thus do not map the ball into
the ball. So composition is not possible on some points off of the variety.
Again the issue with this example is that the varieties have infinitely many irreducible components.
We do not know of any examples with finitely many irreducible components in a finite dimensional ball
where multiplier biholomorphism does not imply isomorphism of the multiplier algebras.
Of course, isomorphism is an equivalence relation.
To show that multiplier biholomorphism is not an equivalence relation in this setting requires 
a counterexample to the hoped-for converse of Theorem~\ref{T:algiso}.
\end{rem}

\section{A class of discs in $\bB_\infty$}  \label{S:Binfty}

We consider a class of embeddings of $\bD$ into $\bB_\infty$, which
were studied in \cite{DRS2}.
Let $(b_n)_{n=1}^\infty \in \ltwo$ with $||(b_n)||_2=1$ and $b_1 \neq 0$. 
Define $f:\bD \to \bB_\infty$ by
\[
 f(z) = (b_1 z, b_2 z^2, b_3 z^3, \ldots) \qfor z \in \bD.
\]
Then $f$ is a biholomorphism with inverse $g = b_1^{-1}Z_1$, and these maps are multipliers.
Moreover the range $V=f(\bD)$ is a variety in the sense of \cite{DRS2} because 
\[
 V = V(\{ b_n z_1^n - b_1^n z_n : n \ge 2\}) = \bigcap_{n\ge2} Z(b_n z_1^n - b_1^n z_n) \cap \bB_\infty .
\]
It is easy to see that any two varieties of this type are multiplier biholomorphic.

Define a kernel on $\bD$ by
\[
  K(z,w) = \frac{1}{1- \langle f(z), f(w) \rangle} \qfor z,w \in \bD,
\]
and let $\H_f$ be the Hilbert function space on $\bD$ with kernel $K$.
Then we can define a linear map 
$U:\H_V \to \H_f$ by $Uh = h\circ f$.
Since
\[
 \ip{k_{f(x)}, k_{f(y)}}_{\H_V} = \frac1{1-\ip{f(x),f(y)}} = \ip{k_x,k_y}_{\H_f} \qforal x,y \in \bD ,
\]
it follows that $U k_{f(x)} = k_x$ extends to a unitary map of $\H_V$ onto $\H_f$.
Hence composition with $f$ determines a unitarily implemented 
completely isometric isomorphism $C_f: \M_V \to \Mult(\H_f)$.
This observation allows us to work with multiplier algebras of Hilbert function spaces
on the disc instead of the algebras $\M_V$. 

Thanks to the special form of $f$, we can write
\[
  K(z,w) = \frac1{1- \sum_{n=1}^\infty |b_n|^2 (z\ol w)^n} =:  \sum_{n=0}^\infty a_n (z \ol w)^n
\]
for a suitable sequence $(a_n)$.
Hence $\H_f$ is a weighted Hardy space.
Set $c_n = |b_n|^2$. 
It was established in \cite{DRS2} that the sequence $(a_n)$ satisfies the recursion
\begin{equation} \label{eqn:a_n_recursion} \tag{3}
  a_0 = 1 \quad \text{ and } \quad
  a_n = \sum_{k=1}^n c_k a_{n-k} \quad \text{ for } n \ge 1.
\end{equation}
Moreover, $a_n \in (0,1]$ for all $n \in \bN$.

\begin{rem}  \label{R:a_n_gen_function}
The coefficients $(a_n)$ can also be determined in the following way. First, note that
as $||(b_n)||_2 = 1$, the function $g$ defined by
\[
  g(z) = \sum_{n=1}^\infty c_n z^n
\]
is holomorphic on $\bD$ and does not take the value $1$ there. Evidently,
\begin{equation} \label{eqn:a_n_gen_function} \tag{4}
  \frac{1}{1 - g(z)} = \sum_{n=0}^\infty a_n z^n \qforal z \in \bD.
\end{equation}
That is, $(a_n)$ is the sequence of Taylor coefficients of $(1-g)^{-1}$ at the origin. 
\end{rem}

The special form of the kernel $K$ allows us to explicitly compute the multiplier norm
of monomials in $\H_f$.

\begin{lem}  \label{L:mult_norm}
Suppose that $\H_f$ is a reproducing kernel Hilbert space on $\bD$ with kernel
\[
  K(z,w) = \sum_{n=0}^\infty a_n (z \ol{w})^n,
\]
where the sequence $(a_n)$ satisfies a recursion as in \eqref{eqn:a_n_recursion} for some sequence
of nonnegative numbers $(c_n)$ with $c_1 \neq 0$. Then
\[
  ||z^n||_{\Mult(\H_f)}^2  = ||z^n||^2_{\H_f}   = \frac{1}{a_n} \qforal n \in \bN.
\]
\end{lem}

\begin{proof}
The assumptions imply that $a_n \neq 0$ for all $n \in \bN$, so by a general result about weighted
Hardy spaces, we have
\[
  ||z^n||^2_{\H_f} = \frac{1}{a_n} .
\]
Therefore for $n \in \bN$,
\[
  ||z^n||^2_{\Mult(\H_h)} = \sup_{k\ge0} \frac{\|z^{n+k}\|_{\H_f}}{\|z^{k}\|_{\H_f}} = \sup_{k \ge0} \frac{a_k}{a_{n+k}} .
\]
Since $a_0 = 1$, it  suffices to show that
\[
  a_k a_n \le a_{n + k} \qforal k,n \in \bN .
\]

The proof of this claim proceeds by induction on $k$. 
The base case holds since $a_0 = 1$.
Assume that $k \ge 1$, and that the assertion has been established for
natural numbers smaller than $k$. Then
\[
  a_k a_n = \sum_{i=1}^k a_{k-i} a_n c_i \le \sum_{i=1}^{n+k} a_{n+k-i} c_i = a_{n+k}
\]
as asserted.
\end{proof}

The results of Section~\ref{S:auto_inv} suggest that we should attempt to verify the properties
\begin{enumerate}
\item for every $\lambda \in V$, the fiber $\pi^{-1}(\lambda) = \{\rho_\lambda\}$, and
\item $\pi (M(\M_V)) \cap \bB_d = V$. 
\end{enumerate} 
We first observe that Proposition~\ref{P:condition(2)} shows that (2) always holds because
the functions $\{b_n z_1^n - b_1^n z_n : n \ge 2 \}$ are polynomials.
In fact, Remark~\ref{R:condition(2)} shows that  $\pi (M(\M_V)) = \ol{V}$. 

We do not know if (1) holds in general.
It does hold for a large class of examples.
In particular, if the ideal of multipliers which vanish at $\lambda$ coincides with $\ol{(z-\lambda) \Mult(\H_f)}$,
then it is clear that any character $\phi\in\pi^{-1}(\lambda)$ must be point evaluation.
We do not have a characterization of when this occurs.
The following result, without the norm closure, will suffice for our current needs.

\begin{lem} \label{L:condition(1)}
Let $ f(z) = (b_1 z, b_2 z^2, b_3 z^3, \ldots)$ for $z \in \bD$, where $\|(b_i)\|_2 \le 1$.
The following assertions are equivalent:
  \begin{enumerate}[label=\normalfont{(\roman*)}]
    \item For every $g \in \M_V$ with $g(0) = 0$, there is $\widetilde g \in \M_V$ such that $g = z_1 \widetilde g$.
    \item For every $g \in \Mult(\H_f)$ with $g(0) = 0$, we have $g /z \in \Mult(\H_f)$.
    \item The sequence $\big(\frac{a_n}{a_{n-1}} \big)_{n \ge 1}$ is bounded.
  \end{enumerate}
\end{lem}

\begin{proof}
  The equivalence of (i) and (ii) follows by an application of the isomorphism
  \begin{equation*}
    \M_V \to \Mult(\H_f), \quad g \mapsto g \circ f.
  \end{equation*}
  Suppose  that (iii) holds. Then
  \begin{equation*}
    D: \H_f \to \H_f, \quad h \mapsto \frac{h - h(0)}{z},
  \end{equation*}
  is a bounded linear map. Indeed $D$ maps $z^n$ to $z^{n-1}$, and $||z^n||^2 = \frac{1}{a_n}$.
  Let $ g \in \Mult(\H_f)$ with $g(0) = 0$. Then for every $h \in \H_f$, we have
  \begin{equation*}
    D M_g h = D (g h) = \frac{g}{z} h.
  \end{equation*}
  This shows that $g/z \in \Mult(\H_f)$ and that $D M_g = M_{g /z}$. Hence, (ii) holds.

  Conversely, suppose that (ii) is satisfied. Then
  \begin{equation*}
    \widetilde D: \Mult(\H_f) \to \Mult(\H_f), \quad g \mapsto \frac{g - g(0)}{z},
  \end{equation*}
  is defined and clearly linear. Since convergence in $\Mult(\H_f)$ implies pointwise
  convergence on $\bD$, we conclude with the help of the closed graph theorem that
  $\widetilde D$ is bounded. In particular,
  \begin{equation*}
    \frac{1}{a_{n-1}} = ||z^{n-1}||^2_{\Mult(\H_f)} = 
    ||\widetilde D z^{n}||^2_{\Mult(\H_f)} \le
    ||\widetilde D||^2 \, ||z^n||^2_{\Mult(\H_f)}
    = ||\widetilde D||^2 \, \frac{1}{a_n}.
  \end{equation*}
  Thus, (iii) holds.
\end{proof}

\begin{cor} \label{C:condition(1)}
Let $ f(z) = (b_1 z, b_2 z^2, b_3 z^3, \ldots)$ for $z \in \bD$, where $\|(b_i)\|_2 \le 1$.
If $\M_V$ is automorphism invariant and $\sup_{n\ge1} \frac{a_n}{a_{n-1}}  < \infty$,
then $\pi^{-1}(\lambda) = \{\rho_\lambda\}$ for every $\lambda\in V$.
\end{cor}

\begin{proof}
The result is immediate for $\lambda = 0$ since every $g\in\M_V$ such that $g(0)=0$
factors as $g=z_1h$ for some $h\in\M_V$.
Thus if $\phi\in\pi^{-1}(0)$, we have $\phi(g) = \phi(z_1)\phi(h) = 0 = \rho_0(g)$.
Hence $\phi=\rho_0$.

Automorphism invariance readily shows that the same holds for every $\lambda\in V$. 
\end{proof}

Suppose now that
\[
  \tilde{f}(z) = (\tilde{b}_1 z, \tilde{b}_2 z^2, \tilde{b}_3z^3, \ldots) \qfor z \in \bD
\]
is another embedding of the disc into $\bB_\infty$ as above, and set $\td{V} = \tilde{f}(\bD)$.
We may define a sequence $(\tilde{a}_n)$ using \eqref{eqn:a_n_recursion} or Remark \ref{R:a_n_gen_function}.
We ask: when are $\M_V$ and $\M_{\td{V}}$ isomorphic?

\begin{prop} \label{P:M_V_iso}
The algebras $\M_V$ and $\M_{\td{V}}$ are isomorphic via the natural
map of composition with $f \circ {\tilde{f}}^{-1}$ if and only if the sequences $(a_n)$ and $(\tilde{a}_n)$ are comparable. 

Suppose that $\M_{\widetilde V}$ satisfies $(1)$\ $\pi^{-1}(\lambda) = \{\rho_\lambda\}$ for every $\lambda\in \widetilde V$
and is automorphism invariant.
Then $\M_V$ is isomorphic to $\M_{\td{V}}$ if and only if the sequences $(a_n)$ and $(\tilde{a}_n)$ are comparable.
In particular, $\M_V$ is isomorphic to $\Hinf$ if and only if the sequence $(a_n)$ is bounded below.
\end{prop}

\begin{proof}
Suppose that $(a_n)$ and $(\tilde{a}_n)$ are comparable. 
The sequence $\{z^n\}$ is a spanning orthogonal sequence,
and Lemma~\ref{L:mult_norm} shows that their norms in $\H_f$ and $\H_{\tilde{f}}$ are comparable.
Thus the identity map is an invertible diagonal operator between $\H_f$ and $\H_{\tilde{f}}$. 
Therefore, $\Mult(\H_f) = \Mult(\H_{\tilde{f}})$, so that
$\M_V$ and $\M_{\td{V}}$ are isomorphic via the natural map.

Conversely, if $\M_{V}$ and $\M_{\td{V}}$ are isomorphic via the natural map, then
$\Mult(\H_f) = \Mult(\H_{\tilde{f}})$.
Therefore the identity map is an isomorphism between these
two semisimple Banach algebras.
Consequently, it is a topological isomorphism. 
So by Lemma \ref{L:mult_norm}, the sequences $(a_n)$ and $(\tilde{a}_n)$ are comparable.

If $\M_V$ is automorphism invariant and satisfies (1), Corollary \ref{C:mob_inv_iso} applies.
So this is equivalent to $\M_V$ being isomorphic to $\M_{\td{V}}$ via any isomorphism.

Note that $H^2$ corresponds to the map $f(z) = (z,0,0,\dots)$; and $a_n = 1$ for all $n\ge1$ because 
\[ \frac1{1-z} = \sum_{n\ge0} z^n .\]
In general, $0<\tilde{a}_n \le 1$, so $(\tilde{a}_n)$ is comparable to $(a_n)$ if and only if it is bounded below.
The last claim now follows from the previous paragraph and the automorphism invariance of $\Hinf = \Mult(H^2)$. 
\end{proof}

In \cite[Example 6.12]{DRS2}, an example was given of a variety $V = f(\bD)$ 
as above such that $\H_f$ is not isomorphic to $H^2$ via the identity map 
(so that $\M_V$ is not similar to $\Hinf$ in the obvious way), 
and the question was raised whether or not $\M_V$ is isomorphic to $\Hinf$. 
The above proposition answers this question, showing that those algebras are not isomorphic. 

The following result gives a criterion for $\M_V$ being isomorphic to $\Hinf$ in terms
of the sequence $(b_n)$ in the definition of the map $f$.

\begin{cor} \label{C:iso_Hinf}
Let $V = f(\bD)$ where $f(z) = (b_1 z, b_2 z^2, b_3 z^3, \ldots)$, $\|(b_n)\|_2 = 1$ and $b_1 \ne 0$.
Then $\M_V$ is isomorphic to $\Hinf$ if and only if
\[
  \sum_{n=1}^\infty n |b_n|^2 < \infty.
\]
\end{cor}

\begin{proof}
We know that $\M_V$ is isomorphic to $\Hinf$ if and only if the sequence $(a_n)$ is bounded below. 
Define
\[
  \mu = \sum_{n=1}^\infty n |b_n|^2 \in (0, \infty].
\]
By the Erd\Humlaut{o}s-Feller-Pollard theorem (see \cite[Chapter XIII, Section 11]{Feller}),
\[
  \lim_{n \to \infty} a_n = \frac{1}{\mu},
\]
where $\infty^{-1} = 0$. 
The theorem is applicable since $|b_1|^2 > 0$. 
Hence, $(a_n)$ is bounded below if and only if this series converges.
\end{proof}

\begin{cor} \label{C:radius}
Let $V$ and $(b_n)$ be as in the previous corollary.
Then if $\M_V$ is not isomorphic to $\Hinf$, the series 
\[
 g(z) = \sum_{n\ge1} |b_n|^2 z^n
 \qand 
 (1-g(z))^{-1} = \sum_{n\ge0} a_n z^n
\]
both have radius of convergence 1.
\end{cor}

\begin{proof}
Since $(b_n)$ is in $\ltwo$, the sequence is bounded, and hence the series for $g$
has radius of convergence at least 1.
If this radius of convergence is $R>1$, then the series $\sum_{n=1}^\infty n |b_n|^2$
converges. So Corollary~\ref{C:iso_Hinf} shows that $\M_V$ is isomorphic to $\Hinf$.
Observe that $g$ is bounded on $\bD$ by $\|(b_n)\|_2^2 =1$. 
In particular, $g(z) \ne 1$ for $z \in \bD$, and thus $(1-g(z))^{-1}$ is defined on $\bD$.
Hence the series for $(1-g(z))^{-1}$ had radius of convergence at least 1.
If this radius of convergence were greater than 1, then the only obstruction to
\[ g(z) = 1 - \frac1{\sum_{n\ge0} a_nz_n} \]
being defined on a disc of radius $R>1$ is that $(1-g(z))^{-1}$ has a zero on $\partial\bD$.
This however implies that $g$ has a pole on the circle, which is impossible because
$g$ is bounded on $\bD$. 
Therefore $\sum_{n\ge0} a_n z^n$ has radius of convergence exactly 1.
\end{proof}

We have seen that not all algebras $\M_V$ are isomorphic to $\Hinf$.
In fact, we will now exhibit a whole scale of mutually non-isomorphic algebras
of this type. To this end, it is again more convenient to work with the algebras
$\Mult(\H_f)$, which are subalgebras of $\Hinf$. 
We ask: which algebras of functions on $\bD$ arise in this way?

\begin{prop}  \label{P:sequence_algebras}
An algebra $\M$ of functions on $\bD$ arises in the way described above
if and only if $\M$ is the multiplier algebra of a Hilbert function space on $\bD$
with kernel $K$ of the form
\[
  K(z,w) = \sum_{n=0}^\infty a_n (z \ol w)^n,
\]
where $a_0 = 1$ and $a_1 \neq 0$, which satisfies the following two properties:
\begin{enumerate}[label=\normalfont{(\arabic*)}]
\item $K$ is a complete Nevan\-linna-Pick kernel.
\item $\dsum_{n=0}^\infty a_n = \infty$.
\end{enumerate}
\end{prop}

\begin{proof}
Suppose that $K$ satisfies the conditions above. Since $K$ is a complete
Nevan\-linna-Pick kernel, the sequence $(c_n)$ defined by
\[
  \sum_{n=1}^\infty c_n z^n = 1 - \frac{1}{\sum_{n=0}^\infty a_n z^n}
\]
is positive by \cite[Theorem 7.33.]{AM}. The last condition guarantees that
\[
  \sum_{n=1}^\infty c_n = \sup_{0 < t < 1} \sum_{n=1}^\infty c_n t^n = 1.
\]
As $a_1 \neq 0$, also $c_1 \neq 0$. Defining $b_n = \sqrt{c_n}$, we see that
$\M$ arises as above (compare Remark \ref{R:a_n_gen_function}).

Conversely, suppose that $\M$ arises as above. Then $\M$ is the multiplier algebra
of a reproducing kernel Hilbert space on the disc whose kernel is of the desired form. 
By \cite[Theorem 7.33.]{AM}, $K$ is a complete Nevan\-linna-Pick kernel. 
Finally,
\[
  \sum_{n=0}^\infty a_n = \sup_{0 < t < 1} \sum_{n=0}^\infty a_n t^n
   = \sup_{0 < t < 1} \frac{1}{1 - \sum_{n=1}^\infty c_n t^n} = \infty
\]
because $\sum_{n=1}^\infty c_n = 1$.
\end{proof}

\begin{eg} \label{E:H_s}
For $s \in \bR$, let $\H_s$ be the reproducing kernel Hilbert space of functions on $\bD$
with kernel
\[
  K(z,w) = \sum_{n=0}^\infty (n+1)^s (z \ol w)^n.
\]
Note that $\H_0$ is the Hardy space, and that $\H_{-1}$ is the Dirichlet space. It is known
that for $s \le 0$, the kernels are complete Nevan\-linna-Pick kernels (see \cite[Corollary 7.41]{AM}).
If $-1 \le s \le 0$, they satisfy the hypotheses of the last proposition (see also \cite[Example 8.8]{AM}).
Consequently, every multiplier algebra $\Mult(\H_s)$ is isomorphic to an algebra $\M_{V_s}$
where $V_s = f_s(\bD)$ is a variety, and 
\[
 f_s(z) = (b_{s,1}z, b_{s,2}z^2, \dots) \qfor z \in \bD
\]
extends to a homeomorphism of $\ol{\bD}$ onto $\ol{V_s}$.
Moreover, each $\H_s$ and thus each $\Mult(\H_s)$ is automorphism invariant (see \cite[Theorem 3.5]{CM}). 
Condition (iii) of Lemma~\ref{L:condition(1)} holds: $\sup_{n\ge1} \frac{(n+1)^s}{n^s} = 2^s < \infty$.
Thus by Corollary~\ref{C:condition(1)}, $\M_{V_s}$ satisfies condition (1).
As we observed, condition (2) always holds.

Therefore Proposition \ref{P:M_V_iso} applies.
Since the sequences $\big( (n+1)^s \big)_{n\ge1}$ are not comparable for distinct values of $s$,
the multiplier algebras $\M_{V_s}$ for $-1 \le s \le 0$ are mutually non-isomorphic. 
In this way, we obtain uncountably many isomorphism classes of algebras $\M_V$.

Consider 
\[
 \ip{ f_s(z), zf_s'(z) } = \sum_{n=1}^\infty n |b_{s,n}|^2 |z|^{2n} .
\]
This converges to a finite limit as $|z|$ tends to 1 if and only if $\sum_{n=1}^\infty n |b_{s,n}|^2 < \infty$, 
which by Corollary~\ref{C:iso_Hinf} holds precisely when 
$\M_{V_s}$ is isomorphic to $\Hinf$, namely when $s=0$.
Moreover, when $s<0$, $f_s$ is not $C^1$ because
\[
 \lim_{|z|\to 1} \|f_s'(z)\|^2 = \lim_{r\to 1} \sum_{n=1}^\infty n^2 r^{2n} |b_{s,n}|^2 = +\infty .
\]

A closely related class of examples are considered in \cite[p.1128--30]{ARS} called 
Besov spaces $B^\sigma_2(\bD)$ for $0 < \sigma < 1/2$. These spaces coincide
as spaces of functions with $H_s$ for $s = -1 + 2\sigma$, although the kernels
are somewhat different.
Not surprisingly, just as for their embeddings, our embeddings are tangential in the sense that
\[ \lim_{x\to 1,\,x\in(0,1)} \frac{1 - \|f_s(xe^{it})\|}{\| f_s(e^{it}) - f_s(xe^{it})\|} = 0 \]
as well. Indeed, using that
\[
  \sum_{n=0}^\infty (n+1)^s x^n \approx \Gamma(1+s) (1-x)^{-1-s}
\]
as $x \to 1$ from below (see \cite[Chap. XIII, p.280, ex. 7]{WW}), we see that
\begin{align*}
  1 - ||f_s(x e^{i t})|| &\sim 1 - ||f_s(x e^{ it})||^2  \\&
  = 1 - \sum_{n=1}^\infty |b_{s,n}|^2 x^{2 n}
  = \Big( \sum_{n=0}^\infty (n+1)^s x^{2 n} \Big)^{-1} \\
  &\sim (1 - x^2)^{1 +s}
  \sim (1-x)^{1 + s}.
\end{align*}
Here, we used the notation $f(x) \sim g(x)$ if $\lim_{x \to 1} f(x) g(x)^{-1} \in (0,\infty)$.
On the other hand,
\begin{align*}
  ||f_s(e^{i t}) - f_s(x e^{i t})||^2 &= 
  \sum_{n=1}^\infty |b_{s,n}|^2 (1 - x^n)^2  \\&
  = 1 - 2 \sum_{n=1}^\infty |b_{s,n}|^2 x^n + \sum_{n=1}^\infty |b_{s,n}|^2 x^{ 2n} \\
  &= 2 \Big( \sum_{n=0}^\infty (n+1)^s x^n \Big)^{-1}
  - \Big( \sum_{n=1}^\infty (n+1)^s x^{2 n} \Big)^{-1}.
\end{align*}
Since
\[
  2 \Big( \sum_{n=0}^\infty (n+1)^s x^n \Big)^{-1} \approx 2 \Gamma(1+s)^{-1} (1-x)^{1 +s}
\]
and
\[
  \Big( \sum_{n=1}^\infty (n+1)^s x^{2 n} \Big)^{-1} \approx \Gamma(1+s)^{-1} (1-x^2)^{1+s}
  \approx \Gamma(1+s)^{-1} 2^{1+s} (1-x)^{1+s},
\]
we have
\[
  ||f_s(e^{i t}) - f_s(x e^{i t})|| \sim (1-x)^{(1+s)/2}.
\]
Thus,
\[ \lim_{x\to 1,\,x\in(0,1)} \frac{1 - \|f_s(xe^{it})\|}{\| f_s(e^{it}) - f_s(xe^{it})\|} = 0. \]
Similarly, for $s= -1$, obtain the same tangential property because
\[
  1 - ||f_s(x e^{i t})|| \sim  
  \Big( \sum_{n=0}^\infty (n+1)^s x^{2 n} \Big)^{-1} \sim -  \log(1-x)^{-1}
\]
and
\begin{align*}
  ||f_s(e^{i t}) - f_s(x e^{i t}) ||^2 &= 
  2 \Big( \sum_{n=0}^\infty (n+1)^s x^n \Big)^{-1}
  - \Big( \sum_{n=1}^\infty (n+1)^s x^{2 n} \Big)^{-1}  \\&
  \sim -  \log(1-x)^{-1}.
\end{align*}

It also follows for $-1 \le s < 0$,
\begin{align*}
 \lim_{x\to 1,\,x\in(0,1)} \frac{\re\ip{f_s(1)-f_s(x),f_s(1)}}{1-x} &= 
 \lim_{x\to 1-} \sum_{n\ge1} |b_{s,n}|^2 \frac{1-x^n}{1-x} \\&= 
 \sum_{n\ge1} n |b_{s,n}|^2 = +\infty
\end{align*}
by Corollary~\ref{C:iso_Hinf}.
\end{eg}

\section{Embedding closed discs}  \label{S:compact}

Next we will consider a class of varieties in $\bB_\infty$ which 
includes the spaces $\H_s$ for $s < -1$.  
Again we define $f:\bD \to \bB_\infty$ by
\[
 f(z) = (b_1 z, b_2 z^2, b_3 z^3, \ldots) \qfor z \in \bD,
\]
with $(b_n)_{n=1}^\infty \in \ltwo$ and  $b_1 \neq 0$.
Here, however, we assume that 
\begin{enumerate}
\renewcommand{\itemsep}{1ex}
\item $||(b_n)||_2 = r<1$, and 
\item $\sum_{n\ge1} |b_n|^2z^n$ has radius of convergence 1. 
\end{enumerate}
Let $V = f(\bD)$. As observed in the previous section, $f$ extends to 
a continuous injection of $\ol{\bD}$ onto $\ol{V}$. 
But because $r<1$, $\ol{V}$ is a compact subset of $\ol{r\bB_\infty} \subset \bB_\infty$.

As we observed in the previous section, $\H_V$ is unitarily equivalent to a
reproducing kernel Hilbert space $\H_f$ on $\bD$ with kernel
\[
  K(z,w) = \frac1{1- \sum_{n=1}^\infty |b_n|^2 (z\ol w)^n} =:  \sum_{n=0}^\infty a_n (z \ol w)^n .
\]
Setting $c_n = |b_n|^2$, we see as in Remark~\ref{R:a_n_gen_function} that
$g(z) = \sum_{n\ge1} c_nz^n$ determines $(a_n)$ by 
\[
  \frac{1}{1 - g(z)} = \sum_{n=0}^\infty a_n z^n \quad \FOR z \in \bD.
\]
Now $(c_n)$ is summable and $\|g\|_\infty = r^2 < 1$, 
so $(1-g)^{-1}$ extends to be continuous on $\ol{\bD}$. 
By hypothesis, the power series for $g(z)$ has radius of convergence 1.
Thus $g$ does not analytically continue across the unit circle---so neither does $(1-g)^{-1}$.
Therefore $\sum_{n=0}^\infty a_n z^n$ also has radius of convergence 1. \vspace{.3ex}
Note that this argument can be reversed: if $\sum_{n=0}^\infty a_n z^n$ has radius of convergence 1,
then so does $\sum_{n\ge1} c_nz^n$. This is much like the proof of Corollary~\ref{C:radius}.

An example of this are the spaces $\H_s$ of Example~\ref{E:H_s} for $s < -1$. 
This space has kernel
\[
  K(z,w) = \sum_{n=0}^\infty (n+1)^s (z \ol w)^n \qfor z,w \in \bD.
\]
Since 
\[
 \sum_{n\ge0} a_n = \sum_{n\ge0} (n+1)^s < \infty,
\]
this doesn't fit into Proposition~\ref{P:sequence_algebras}.
However, this series has radius of convergence 1, so by the previous paragraph,
it fits into the framework of this section.

The fact that $f$ has radius of convergence 1 means that there is no natural extension 
of $V$ beyond $\ol{V}$ to something that looks like a variety in the classical sense. 
In finite dimensions, no variety in $\bB_d$ can be compact \cite[Theorem 14.3.1]{Rudin}. 
Nevertheless, it turns out that while $V$ is not a variety, its compact closure $\ol{V}$ is a variety!

\begin{lem} \label{L:compact_variety}
If $(b_n)$ and $f$ are defined as above, then $\ol{V} = f(\ol{\bD})$ is the common zero locus
of the polynomials $\{ b_n z_1^n - b_1^n z_n : n \ge 2 \}$; that is,
\[
  \ol{V} = V(\{ b_n z_1^n - b_1^n z_n : n \ge 2 \}) .
\]
\end{lem}

\begin{proof}
Note that every point in $\ol{V}$ is a zero of the polynomials $b_n z_1^n -b_1^n z_n$.
Conversely, if $\Bz = (z_1,z_2,\ldots)$ satisfies these equations, then setting $z = z_1/b_1$,
we find that
\[
  z_n = b_n z^n \qforal n \in \bN .
\]
Since $(z_1,z_2, \ldots)$ is a point in $\ltwo$, we have
\[
  \infty > \sum_{n=1}^\infty |z_n|^2 = \sum_{n=1}^\infty |b_n|^2 |z|^{2 n}.
\]
As the series on the right has radius of convergence 1, 
it follows that $|z| \le 1$. 
Hence $z \in \ol{\bD}$ and $\Bz = f(z)$ belongs to $\ol{V}$.
\end{proof}
\begin{rem}
$\ol{V}$ is the minimal variety containing $V$. 
Hence $\H_V = \H_{\ol{V}}$, and $\M_V$ can be naturally identified with $\M_{\ol{V}}$ 
(see \cite[Proposition 2.2]{DRS2}). 
It is $\ol{V}$, not $V$, which fits into the framework developed in \cite{DRS2}. 
\end{rem}

Another curious property of these spaces is the following.
We believe that this result is well known, but we do not have a reference.

\begin{lem} \label{L:cnts_fns}
If $(b_n)$ and $f$ are defined as above, then $\H_{\ol{V}}$ and $\M_{\ol{V}}$ consist of 
continuous functions on $\ol{V}$.
\end{lem}

\begin{proof}
The constant function $1$ is in $\H_V = \H_{\ol{V}}$, thus $\M_{\ol{V}} \subset \H_{\ol{V}}$, 
so it suffices to prove the claim for $\H_{\ol{V}}$. 
It is convenient to first work with $\H_f$ instead of $\H_{\ol{V}}$.
By Lemma~\ref{L:mult_norm}, we have $\|z^n\|_{\H_f} = a_n^{-1/2}$.
Every function $h(z)$ in $\H_f$ has an orthonormal expansion
\[
 h(z) = \sum_{n\ge0} d_n z^n 
 \quad\text{where}\quad
 \sum_{n\ge0} \frac{|d_n|^2}{a_n} = \|h\|_{\H_f}^2 < \infty .
\]
This series converges uniformly on $\ol{\bD}$ since
\begin{align*}
 \sum_{n\ge0} |d_n| &= \sum_{n\ge0} \frac{|d_n|}{\sqrt{a_n}} \sqrt{a_n} \\&
 \le \Big(  \sum_{n\ge0} \frac{|d_n|^2}{a_n} \Big)^{1/2} \Big(  \sum_{n\ge0} a_n \Big)^{1/2} \\&
 = \|h\|_{\H_f} \Big( \frac 1{1-g(1)} \Big)^{1/2} = \frac{\|h\|_{\H_f}}{\sqrt{1-r^2}} .
\end{align*}
Therefore every $h\in\H_f$ is continuous. As $f$ is a homeomorphism 
of $\ol{\bD}$ onto $\ol{V}$, this transfers to $\H_{\ol{V}}$.
\end{proof}

Let $\delta:\ol{V} \to M(\M_{\ol{V}})$ be the map taking $v \in \ol{V}$ to the character $\rho_v$ which evaluates multipliers at $v$. 
We do not know if $\delta$ is always a homeomorphism.
On the other hand, we know of no example of a compact variety $\ol{V}$ contained in the open ball $\bB_\infty$ 
where the maximal ideal space of the multiplier algebra is not homeomorphic to $\ol{V}$. 
Shields  \cite[section 9]{Shields} asks a similar question in the context of spaces of weighted shifts. 
He answers the question positively when the algebra is strictly cyclic, in which case the multiplier algebra
coincides with the Hilbert space (as functions). We can use his result here.

\begin{lem} \label{L:gelfand_comp_var}
  Suppose that
  \begin{equation*}
    \sup_{n\ge1} \sum_{k=0}^n \Big( \frac{a_k a_{n-k}}{a_n} \Big) < \infty.
  \end{equation*}
  Then the natural injection $\delta$ of $\ol{V}$ into $M(\M_{\ol{V}})$ is a homeomorphism. In particular, this
  is the case if $\overline{V}$ arises from $H_s, s < -1$.
\end{lem}

\begin{proof}
  The results in Section 9 of \cite{Shields} show that the operator $M_z$ on $\H_f$ is strictly cyclic if the supremum
  is finite,
  hence the Gelfand space of $\Mult(\H_f)$ is the closed unit disc. It follows that $\delta$ is a homeomorphism. In the
  case of $\H_s, s < -1$, Example 1 after Proposition 33 in \cite{Shields} shows that the supremum is finite.
\end{proof}

Now we can establish isomorphism results for this family of compact varieties
that parallel the results of the previous section.

\begin{thm} \label{T:compact_var_iso} 
Let 
\[ f(z) = (b_1z,b_2z^2,b_3z^3,\dots) \qand  \tilde{f}(z) = (\tilde{b}_1z,\tilde{b}_2z^2,\tilde{b}_3z^3,\dots) \]
be functions of $\ol{\bD}$ into $\bB_\infty$,  with 
\[ \|(b_n)\|_2 = r < 1, \quad  \|(\tilde{b}_n)\|_2 = r'<1 \qand  b_1\tilde{b}_1 \ne 0 \] 
such that the series 
\[ \sum_{n\ge1} |b_n|^2 z^n \qqand  \sum_{n\ge1} |\tilde{b}_n|^2 z^n \]
both have radius of convergence $1$.
Let $\ol{V} = f(\ol{\bD})$ and $\ol{\td{V}} = \tilde{f}(\ol{\bD})$.

\begin{enumerate}[label=\normalfont{(\arabic*)}]
\item $\M_{\ol{V}}$ and $\M_{\ol{\td{V}}}$ are isomorphic via the natural map if and only if
the sequences $(a_n)$ and $(\tilde{a}_n)$ are comparable.

\item Suppose that $\M_{\ol{V}}$ satisfies the hypothesis of Lemma \ref{L:gelfand_comp_var}: 
\[ 
 \sup_{n\ge1} \sum_{k=0}^n \Big( \frac{a_k a_{n-k}}{a_n} \Big) < \infty 
 \quad\text{where}\quad \sum_{n\ge0} a_n z^n = \frac1{1-\sum_{n\ge1} |b_n|^2z^n},
\] 
and is automorphism  invariant.
Assume that $\M_{\ol{V}}$ is isomorphic to $\M_{\ol{\td{V}}}$.
Then the restriction $F$ of $ \phi^*$ to $\ol{\td{V}}$ is a homeomorphism of $\ol{\td{V}}$ onto $\ol{V}$ 
which is holomorphic on $\td{V}$ and $\phi(h) = h \circ F$. 
There is a M\"obius map $\theta$ so that the following diagram commutes:
\[
  \xymatrix{ \M_{\ol V} \ar[r]^{\phi} \ar[d]_{C_f} & \M_{\ol{\td{V}}} \ar[d]^{C_{\tilde{f}}} \\ \AD \ar[r]_{C_\theta} & \AD }
\]
Moreover, they are isomorphic via the natural map of composition with $G = f\circ f^{\prime -1}$. 
\end{enumerate}
\end{thm}

\begin{proof}
(1)  This follows as in Proposition~\ref{P:M_V_iso}.

(2) Since $\phi$ is a continuous isomorphism, $\phi^*$ yields a homeomorphism of
the maximal ideal spaces. 
By Lemma~\ref{L:gelfand_comp_var}, $M(\M_{\ol{V}}) = \delta(\ol{V}) \simeq \ol{V}$.
So we identify $M(\M_{\ol{V}})$ with $\ol{V}$.
For the other algebra, we have that $\ol{\td{V}}$ is identified with $\delta(\ol{\td{V}})$ as a subset of
$M(\M_{\ol{\td{V}}})$.
Let $F:\ol{\td{V}}\to\ol{V}$ be the restriction of $\phi^*$ to this copy of $\ol{\td{V}}$.
The argument in the proof of Theorem~\ref{T:iso_discs} works here too, to show that $F$
is holomorphic on $\td{V}$. 
Now
\[
 \phi(h)(\tilde{v}) = \phi^*(\rho_{\tilde{v}})(h) = \rho_{F(\tilde{v})}(h) = (h \circ F)(\tilde{v}) 
 \qfor h \in \M_{\ol{V}} \, \AND \, \tilde{v} \in \ol{\td{V}} .
\]
Thus $\phi = C_F$ is a composition operator.

By an adaptation of \cite[Section 11.3]{DRS1} as in the proof of Theorem~\ref{T:iso_discs}, 
the fact that $\phi$ is implemented by composition implies that $\phi$ is weak-$*$ continuous. 
And the argument continues to conclude that $\phi^{-1}$ is also weak-$*$ continuous. 
In particular, $(\phi^{-1})^*$ takes point evaluations to point evaluations. 
As this map is the inverse of $F$, we deduce that $F$ maps onto $\ol{V}$; 
and hence $M(\M_{\ol{\tilde V}}) = \ol{\tilde V}$.

Repeat the proof of Corollary~\ref{C:Mobius} to get the commutative diagram.
The only change is that, since the multipliers are continuous by Lemma~\ref{L:cnts_fns},
the range is considered as a subalgebra of the disc algebra $\AD$, 
rather than in the larger algebra $\Hinf$.
Since $\M_{\ol{V}}$ is automorphism invariant, we may apply the 
automorphism for $\theta^{-1}$ to obtain the natural map as in Proposition~\ref{P:M_V_iso}.
\end{proof}

\begin{eg}
The spaces $\H_s$ for $s<-1$ yield an uncountable family of varieties in $\bB_\infty$
which are homeomorphic to $\ol{\bD}$.
Their multiplier algebras are automorphism invariant (see \cite[Theorem 3.5]{CM}) and they satisfy
the hypothesis of Lemma \ref{L:gelfand_comp_var}.
These sequences are not comparable for different values of $s$.
Thus by Theorem~\ref{T:compact_var_iso}, they have non-isomorphic multiplier algebras.
\end{eg}

\section{Interpolating sequences}  \label{S:interpolating}

We finish the treatment of the algebras $\M_{\ol{V}}$ of the previous section
by showing that under the assumptions of Lemma~\ref{L:gelfand_comp_var} these algebras are not isomorphic to an algebra of the type 
$\M_W$ for any variety $W$ whose closure meets the boundary of the ball.
This result should not be surprising, as isomorphism of the algebras yields
a homeomorphism of the maximal ideal spaces. In the setting of Lemma~\ref{L:gelfand_comp_var}
the maximal ideal space is homeomorphic to $\ol{\bD}$.
The reader may suspect that this is never the case when $\ol{W}$ intersects the boundary.

The way we will establish this is by showing that any sequence in the ball that
converges to the boundary contains an interpolating subsequence.
It then follows that $\M_W$ has $\linf$ as a quotient, and hence its
maximal ideal space contains a copy of the Stone-\v Cech compactification $\beta \bN$ of $\bN$.
In particular, it is not metrizable; so it is not the unit disc. 
We were not able to show, without imposing any special assumptions, that an algebra $\M_{\ol{V}}$ as in Section~\ref{S:compact} can never be isomorphic to an algebra of the type occuring in Section ~\ref{S:Binfty}. 

A sequence $(x_n)$ in $\bB_\infty$ is an \textit{interpolating sequence} for $\M_\infty$
if for every sequence $(a_n) \in \linf$, there is a multiplier $h \in \M_\infty$
such that $h(x_n)=a_n$.
The multiplier algebras considered here are all of the form $\M_V$,
where $V$ is a variety in $\bB_\infty$.
These are (complete) quotients of $\M_\infty$ via the restriction map.
So any sequence in $V$ is interpolating for $\M_V$ if and only if it is interpolating for $\M_\infty$.

\begin{prop} \label{P:interpolating_sequences}
Let $(z_n)$ be a sequence in $\bB_\infty$ such that $\lim_{n \to \infty} |z_n| = 1$.
Then $(z_n)$ contains a subsequence which is interpolating for $\M_\infty$.
\end{prop}

\begin{proof}
Fix $r \in (0,1)$.
We wish to show that there is a subsequence $(z_{n_k})$ of $(z_n)$ such that for every sequence
$(w_k) \in \ell^\infty$ of norm at most $r$, there is a multiplier $\phi\in \Mult(\H)$
of norm at most $1$ such that $\phi(z_{n_k}) = w_k$. We will recursively
construct the subsequence $(z_{n_k})$ such that for each $k$ and for all $w = (w_i) \in \ell^\infty$
with $||w|| \le r$, the $k \times k$ matrix
\begin{align*}
  A_k(w) = \Big[ (1-w_i \ol w_j) K(z_{n_i},z_{n_j}) \Big]_{i,j=1}^k
\end{align*}
is positive and invertible.
Once we have achieved this, the Nevan\-linna-Pick property yields,
for each $w \in \ell^\infty$ with $||w|| \le r$  and any positive integer $k$, 
the existence of a multiplier $h_k$ of norm at most $1$ such
that $h_k(z_{n_i}) = w_i$ for $1 \le i \le k$. 
Any weak*-cluster point $h$ of the sequence $(h_k)$
will then satisfy $h(z_{n_i}) = w_i$ for all $i \in \bN$.

We begin the construction by setting $z_{n_1} = z_1$. 
Suppose that $k \ge 2$ and that $z_{n_1}, \ldots, z_{n_{k-1}}$ have already been constructed. 
Given $w = (w_i) \in \ell^\infty$ with $||w|| \le r$, we set $v_{i j} = 1-w_i \ol{w_j}$. 
For $z \in \bB_\infty$, we consider the matrix $A(w,z)$ defined by
\[
  \begin{bmatrix}
  v_{1, 1} K(z_{n_1}, z_{n_1}) & \cdots & v_{1, k-1} K(z_{n_1}, z_{n_{k-1}}) & v_{1, k}  K(z_{n_1}, z) \\
   \vdots & \ddots & \vdots & \vdots  \\
  v_{k-1, 1} K(z_{n_{k-1}}, z_{n_1}) & \cdots & v_{k-1, k-1} K(z_{n_{k-1}},z_{n_{k-1}})  & v_{k-1, k} K(z_{n_{k-1}}, z) \\
  v_{k, 1} K(z, z_{n_1}) & \cdots & v_{k, k-1} K(z, z_{n_{k-1}}) & v_{k, k}  K(z, z)
  \end{bmatrix}.
\]
Observe that the first $(k-1) \times (k-1)$ minor equals $A_{k-1}(w)$, which is positive and invertible for all
choices of $w$ with $||w|| \le r$ by our
recursive assumption.
By Sylvester's criterion, it therefore suffices to show
that there exists $z_{n_k}$ with $n_k > n_{k-1}$
such that $\det(A(w,z_{n_k})) > 0$ for all such $w$. To see that this is possible, note that
\[
  \lim_{n \to \infty} K(z_n,z_n) = \lim_{n \to \infty} \frac1{1-\|z_n\|^2} = \infty .
\]
On the other hand, each $K(z_i,z)$ is bounded.
Moreover, by compactness of the unit ball in finite-dimensional spaces,
there exists $\delta > 0$ such that
\[
  \det(A_{k-1}(w)) > \delta
\]
for all $w$ with $||w|| \le r$.
Thus, in the expansion of the determinant of $A_n(w,z)$ along the last row, 
there is one term
\[
  |v_{k k} K(z,z) \det(A_{k-1}(w))| \ge (1-r^2) \delta K(z,z),
\]
which tends to infinity as $z \to 1$ uniformly in $w$, whereas all other terms are uniformly bounded.
Therefore the determinant is eventually strictly positive on the whole $r$-ball.
This establishes the existence of the desired point $z_{n_k}$, and thus finishes the recursive construction.
\end{proof}

\begin{cor} \label{C:stoneCech}
If $W$ is a variety in the ball $\bB_d$ for $d \le \infty$ such that $\ol{W}$ intersects 
the boundary of the ball, then $\linf$ is a quotient of $\M_W$ and $M(\M_W)$
contains a copy of $\beta \bN$.
\end{cor}

\begin{proof}
Proposition~\ref{P:interpolating_sequences} shows that $W$ contains an interpolating sequence.
The restriction map to this sequence is the desired quotient onto $\linf$.
Hence $M(\linf)$, which is homeomorphic to $\beta\bN$, embeds as a closed subset of $M(\M_W)$.
\end{proof}

Thus we obtain the desired consequence.

\begin{prop}
Let $\ol{V}$ be a compact variety as considered in Theorem~$\ref{T:compact_var_iso}(2)$, 
and let $\td{V}$ be a variety as considered in section~$\ref{S:Binfty}$. 
Then there is no unital surjective algebra homomorphism from $\M_{\ol{V}}$ onto $\M_{\td{V}}$.
In particular, they are not isomorphic.
\end{prop}

\begin{cor}
The Hilbert spaces $\H_s$ have non-isomorphic multiplier algebras for distinct $s \le 0$.
\end{cor}

\bibliographystyle{amsplain}

\end{document}